\documentclass[final,12pt,a4paper]{amsart} 

\usepackage{graphicx}
\usepackage[all]{xy}
\usepackage{placeins}
\usepackage{enumitem}
\usepackage{amssymb}
\usepackage{latexsym}
\usepackage{amsmath}
\usepackage{mathrsfs}
\usepackage{array,booktabs}

\usepackage[notref, notcite]{showkeys}  
\usepackage[dvipsnames]{xcolor}

\usepackage{hyperref}
\hypersetup{
    colorlinks,
    linkcolor={red!50!black},
    citecolor={blue!50!black},
    urlcolor={blue!80!black}
}

\newcommand{\hyref}[2]{ \hyperref[#2]{#1~\ref*{#2}} } 

\usepackage{fullpage}

\newcommand{\coloneqq}{\mathrel{\mathop:}=}

\theoremstyle{plain}
\newtheorem{theorem}{Theorem}[section]
\newtheorem{lemma}[theorem]{Lemma}
\newtheorem{corollary}[theorem]{Corollary}
\newtheorem{proposition}[theorem]{Proposition}
\newtheorem{introtheorem}{Theorem}

\theoremstyle{definition}
\newtheorem{remark}[theorem]{Remark}
\newtheorem{definition}[theorem]{Definition}
\newtheorem{notation}[theorem]{Notation}
\newtheorem{setup}[theorem]{Setup}
\newtheorem*{convention}{Convention}
\newtheorem*{acknowledgments}{Acknowledgments}

\newcommand{\Case}[1]{\smallskip\noindent\emph{Case #1:}}

\newcommand{\sA}{\mathsf{A}}
\newcommand{\sB}{\mathsf{B}}
\newcommand{\sC}{\mathsf{C}}
\newcommand{\sD}{\mathsf{D}}
\newcommand{\sG}{\mathsf{G}}
\newcommand{\sK}{\mathsf{K}}
\newcommand{\sS}{\mathsf{S}}
\newcommand{\sT}{\mathsf{T}}

\DeclareMathAlphabet{\mathpzc}{OT1}{pzc}{m}{it}

\newcommand{\bN}{\mathbb{N}}
\newcommand{\bZ}{\mathbb{Z}}

\newcommand{\mcD}{\mathcal{D}}
\newcommand{\mcH}{\mathcal{H}}

\newcommand{\mcS}{\mathcal{S}}
\newcommand{\mcX}{\mathcal{X}}
\newcommand{\mcY}{\mathcal{Y}}
\newcommand{\mcZ}{\mathcal{Z}}

\newcommand{\cl}[1]{\mathbf{#1}} 

\newcommand{\kk}{{\mathbf{k}}} 

\newcommand{\Db}{\sD^b}
\newcommand{\T}{\sT}
\newcommand{\Tc}{\sT^{\mathrm{c}}} 
\newcommand{\Kc}{\sK^{\mathrm{c}}} 
\newcommand{\KProjL}{\sK(\mathsf{Proj}\mbox{-}\Lambda)}
\newcommand{\KbprojL}{\sK^b(\mathsf{proj}\mbox{-}\Lambda)}
\newcommand{\KprojL}{\sK^{-,b}(\mathsf{proj}\mbox{-}\Lambda)}
\newcommand{\Kminus}[1]{\sK^{-,b}(\proj{#1})}
\newcommand{\strings}[1]{\mathsf{Str}\mbox{-}#1}

\newcommand{\Mod}[1]{\mathsf{Mod}\mbox{-}#1} 
\renewcommand{\mod}[1]{\mathsf{mod}\mbox{-}#1} 
\newcommand{\Ab}{\mathsf{Ab}} 
\newcommand{\Coh}[1]{\mathsf{Coh}( #1)} 
\newcommand{\ind}[1]{\mathsf{ind}(#1)}
\newcommand{\proj}[1]{\mathsf{proj}\mbox{-}#1}
\newcommand{\Proj}[1]{\mathsf{Proj}\mbox{-}#1}
\newcommand{\add}[1]{\mathsf{add}(#1)}
\newcommand{\Pinj}[1]{\mathsf{Pinj}(#1)} 
\newcommand{\Inj}[1]{\mathsf{Inj}(#1)} 
\newcommand{\Abs}[1]{\mathsf{Abs}\mbox{-}#1} 
\newcommand{\Def}[1]{\langle #1 \rangle} 
\newcommand{\fun}[1]{\mathsf{fun}(#1)}

\DeclareMathOperator{\coker}{\mathrm{coker}}
\DeclareMathOperator{\image}{\mathrm{im}}
\DeclareMathOperator{\Hom}{\mathrm{Hom}}
\DeclareMathOperator{\Ext}{\mathrm{Ext}}
\DeclareMathOperator{\id}{{\mathrm{id}}}
\DeclareMathOperator{\End}{{\mathrm{End}}}
\DeclareMathOperator{\modulo}{\mathrm{mod}}
\DeclareMathOperator{\gldim}{\mathrm{gldim}}
\DeclareMathOperator{\KG}{\mathrm{KGdim}} 
\DeclareMathOperator{\CB}{\mathrm{CB}\mbox{-}\mathrm{rank}} 
\DeclareMathOperator{\supp}{\mathrm{supp}}
\DeclareMathOperator{\length}{\mathrm{length}}

\newcommand{\Zg}{\mathrm{Zg}} 
\newcommand{\op}{^\mathrm{op}} 
\newcommand{\fp}{^\mathrm{fp}} 
\newcommand{\orth}{^{\perp}} 
\newcommand{\LLambda}{\Lambda(r,n,m)}
\newcommand{\iLLambda}{\Lambda(n,n,m)}
\newcommand{\thick}{\mathsf{thick}}
\newcommand{\homfrom}[1]{\mathsf{H}^{+}(#1)}
\newcommand{\infhomfrom}[1]{\mathsf{H}^{+}_{\infty}(#1)}
\newcommand{\homto}[1]{\mathsf{H}^{-}(#1)}
\newcommand{\infhomto}[1]{\mathsf{H}^{-}_{\infty}(#1)}
\newcommand{\mctZ}{\widetilde{\mathcal{Z}}}
\newcommand{\ccdots}{\cdot\!\cdot\!\cdot} 

\newcommand{\into}{\hookrightarrow}
\newcommand{\too}{\longrightarrow}
\newcommand{\rightlabel}[1]{\stackrel{#1}{\longrightarrow}}
\newcommand{\rightiso}{\rightlabel{\sim}}
\newcommand{\bij}{\stackrel{1-1}{\longleftrightarrow}}

\newcommand{\tri}[3]{#1\rightarrow #2\rightarrow #3\rightarrow \Sigma #1}
\newcommand{\trilabels}[6]{#1\stackrel{#4}{\longrightarrow} #2\stackrel{#5}{\longrightarrow} #3\stackrel{#6}{\longrightarrow} \Sigma #1}

\newcommand{\hocolim}{\underrightarrow{\mathsf{holim}}\,}
\newcommand{\colim}{\underrightarrow{\mathsf{lim}}\,}
\newcommand{\F}{\overrightarrow{F}}

\renewcommand{\phi}{\varphi}
\renewcommand{\epsilon}{\varepsilon}

\usepackage{tikz}
\usetikzlibrary{arrows,decorations.pathmorphing,decorations.pathreplacing}
\tikzset{vertex/.style={circle,fill=black,inner sep=1pt,outer sep=2pt},
         tinyvertex/.style={font=\scriptsize,minimum size=6pt},
         smallvertex/.style={inner sep=1pt, font=\small},
         >=stealth',
         leadsto/.style={-angle 90,decorate,decoration=snake,very thick},
         cut/.style={decorate,decoration=saw,very thick}}
\tikzset{
    partial ellipse/.style args={#1:#2:#3}{
        insert path={+ (#1:#3) arc (#1:#2:#3)}
    }
}

\begin{document}

\title[Ziegler spectrum]{The Ziegler spectrum for derived-discrete algebras}

\author{Kristin Krogh Arnesen}
\address{Institutt for Matematiske Fag, Norges Teknisk-Naturvitenskapelige Universitet, N-7491 Trondheim, Norway.}
\email{kristin.arnesen@math.ntnu.no}

\author{Rosanna Laking}
\address{School of Mathematics, The University of Manchester, Oxford Road, Manchester, M13 9PL, United Kingdom.}
\email{rosanna.laking@manchester.ac.uk}

\author{David Pauksztello}
\address{School of Mathematics, The University of Manchester, Oxford Road, Manchester, M13 9PL, United Kingdom.}
\email{david.pauksztello@manchester.ac.uk}

\author{Mike Prest}
\address{School of Mathematics, The University of Manchester, Oxford Road, Manchester, M13 9PL, United Kingdom.}
\email{mprest@manchester.ac.uk}



\date{\today}

\begin{abstract}
Let $\Lambda$ be a derived-discrete algebra. We show that the Krull-Gabriel dimension of the homotopy category of projective $\Lambda$-modules, and therefore the Cantor-Bendixson rank of its Ziegler spectrum, is $2$, thus extending a result of Bobi\'nski and Krause \cite{BK}. We also describe all the indecomposable pure-injective complexes and hence the Ziegler spectrum for derived-discrete algebras, extending a result of Z.~Han \cite{HanThes}. Using this, we are able to prove that all indecomposable complexes in the homotopy category of projective $\Lambda$-modules are pure-injective, so obtaining a class of algebras for which every indecomposable complex is pure-injective but which are not derived pure-semisimple. 
\end{abstract}

\maketitle

{\small
\setcounter{tocdepth}{1}
\tableofcontents
}

\section*{Introduction}
\addtocontents{toc}{\protect{\setcounter{tocdepth}{1}}}   

We prove that if $\Lambda$ is a derived-discrete algebra then the Krull-Gabriel dimension of the homotopy category, $\KProjL$, of projective modules is 2, which is therefore also the Cantor-Bendixson rank of the Ziegler spectrum of $\KProjL$.  This extends a result of Bobi\'nski and Krause who computed this dimension for the derived category.  We also show that every indecomposable complex of $\KProjL$ is a homotopy string complex, in particular is pure-injective and hence a point of the Ziegler spectrum, which we also describe.  Our proof that there are no further indecomposable objects relies heavily on that description.

The axioms of triangulated categories are an abstraction of the structural properties enjoyed by derived categories in representation theory and algebraic geometry, and stable homotopy categories in algebraic topology. Thus, triangulated categories provide a common language for a broad part of modern mathematics. In representation theory triangulated categories play a central role in tilting theory, which is a technique to enable effective comparisons between the representation theory of different algebras. 

Whilst useful, triangulated categories often have complicated structure. One can gain some understanding by studying various aspects. For example, one may try to classify the thick or localising subcategories.  In this article we focus on another method of simplification: identifying the indecomposable pure-injective objects and the space, the Ziegler spectrum, originally arising from the model theory of modules, that they form. 
Methods involving purity and the Ziegler spectrum were extended to compactly generated triangulated categories by Beligiannis, Krause and others; see for example \cite{Belig, BenGnac, KraTel, NeeReview} and the references therein.

Let $\T$ be a compactly generated triangulated category (see Section~\ref{background}). We shall focus on three essentially equivalent approaches to studying the structure of $\T$; namely understanding the objects in the following diagram.
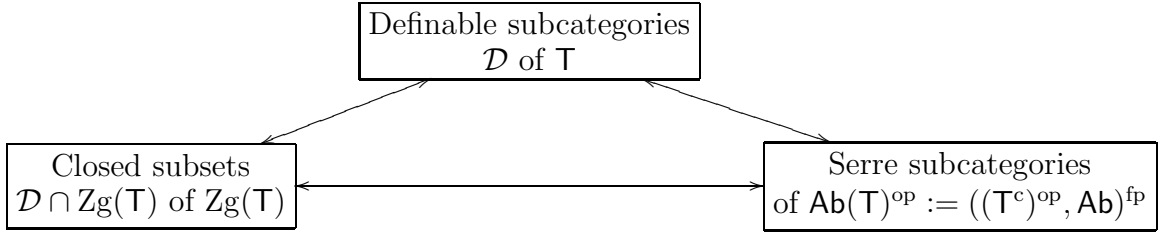
\[
\xymatrix{
 & *+[F]{\txt{Definable subcategories \\$\mcD$ of $\T$}} \ar@{<->}[dl] \ar@{<->}[dr] & \\
*+[F]{\txt{Closed subsets \\ $\mcD \cap \Zg(\T)$ of $\Zg(\T)$}} \ar@{<->}[rr] & & *+[F]{\txt{Serre subcategories \\of $\Ab(\T)\op \coloneqq ((\Tc)\op, \Ab)^{\mathrm{fp}}$}}
}
\]
Here $\Zg(\T)$, the \emph{Ziegler spectrum} of $\T$, is a topological space whose points are the (isomorphism classes of) indecomposable pure-injective objects of $\T$, and $\Ab(\T)\op \coloneqq ((\Tc)\op, \Ab)^{\mathrm{fp}}$ is the category of finitely presented contravariant functors from the subcategory $\Tc$ of compact objects, to abelian groups; this may be regarded as an abelian approximation of $\T$.
We refer to Section~\ref{background} for precise definitions of the concepts involved in the diagram above.

Associated to a topological space is its \emph{Cantor-Bendixson rank} and associated to an abelian category is its \emph{Krull-Gabriel dimension}. Applied to $\Zg(\T)$ and $\Ab(\T)\op$, each of these may be regarded as a measure of the complexity of the triangulated category $\T$. Under suitable conditions, the process of computing the Cantor-Bendixson rank of $\Zg(\T)$ and the Krull-Gabriel dimension of $\Ab(\T)\op$ proceed in tandem. Indeed, there is a strong link between the two computations. In this article, we shall use this link to study the homotopy category of projective left $\Lambda$-modules $\KProjL$ for $\Lambda$ a derived-discrete algebra.

Derived-discrete algebras were introduced by Vossieck in \cite{Vossieck} and occupy  a position between finite and tame representation type: $\Lambda$ is derived-discrete if and only if up to shift there are infinitely many indecomposable objects in $\Db(\mod{\Lambda})$ but there are no 1-parameter families; see \cite[\S 1.1]{Vossieck}. It was shown in \cite{Vossieck} that $\Lambda$ is derived-discrete if and only if it is piecewise hereditary of Dynkin type or a gentle one-cycle algebra satisfying a certain condition on its quiver. From now on when we refer to $\Lambda$ as a derived-discrete algebra we shall exclude the case that $\Lambda$ is piecewise hereditary of Dynkin type since these categories are very well understood. The philosophy behind studying derived-discrete algebras is that they provide a natural laboratory of computationally feasible but nontrivial derived categories in which to explore derived representation theory, and they provide a template for further study of derived categories of gentle algebras, these being important in cluster-tilting theory. As such, they have been the focus of much recent interest \cite{Bo2,BGS,BK,BPP1,BPP2,HanZhang,HanThes,Qin,Vossieck}.

From now on $\Lambda$ will be a derived-discrete algebra over an algebraically closed field $\kk$. Our first main result refers to the bottom vertices of the diagram above:

\begin{introtheorem}\label{intro:zgrank}
Let $\Lambda$ be a derived-discrete algebra.  Then $\KG(\KProjL)=2$, which is also the Cantor-Bendixson rank of the Ziegler spectrum of $\KProjL$.
\end{introtheorem}

\noindent
This extends the result of Bobi\'nski and Krause in \cite{BK}, who compute the Krull-Gabriel dimension of the abelianisation of the perfect complexes $\Ab(\sD(\Mod{\Lambda}))\op$, getting
\[
\KG(\sD(\Mod{\Lambda})) = 1 \text{ if } \gldim \Lambda = \infty \text{ and }
\KG(\sD(\Mod{\Lambda})) = 2 \text{ if } \gldim \Lambda < \infty.
\]

We then use the correspondences in the diagram to obtain a description of the points of the Ziegler spectrum and its topology.

The indecomposable objects of $\Db(\mod{\Lambda})$ were classified in \cite{BGS} and the morphisms studied in \cite{BPP1}. This fits inside a more general classification,  due to Bekkert and Merklen \cite{BM}, of indecomposable complexes in the bounded derived category of a gentle algebra and a more general description, in \cite{ALP}, of morphisms between the indecomposable complexes.
We build on these results, describing all the indecomposable complexes of $\KProjL$.  We show that every indecomposable is pure-injective, hence a point of the Ziegler spectrum of $\KProjL$.

\begin{introtheorem} \label{intro:indecomposable}
Let $\Lambda$ be a derived-discrete algebra. Then each indecomposable complex in $\KProjL$ is both:
\begin{enumerate}
\item a (possibly infinite) homotopy string complex;
\item pure-injective.
\end{enumerate}
Moreover, the homotopy string complexes, of which we have a complete list, are all the indecomposable pure-injective objects of $\KProjL$.
\end{introtheorem}

We also describe the Cantor-Bendixson stratification of the Ziegler spectrum, identifying the simple functors and the corresponding points relatively isolated by them.  In the case of the derived-discrete algebras $\Lambda(n,n,0)$, $n\geq 1$ the Ziegler spectrum had already been described by Z. Han \cite{HanThes} using covering theory.  Our methods are very different, replying heavily on the description of morphisms in \cite{ALP} and techniques around the Ziegler spectrum and the functor category.  
We will require ideas and results coming from various directions; some we summarise and for all we give references to where the details can be found.

In Section~\ref{background} we describe purity in a compactly generated triangulated category $\T$.  We then go on to define the Ziegler spectrum, $\Zg(\T)$, of $\T$, both internally and through the restricted Yoneda functor to the functor category on the compact objects of $\T$.  We also describe the corresponding two, anti-equivalent, categories of coherent functors.  Localisation of abelian categories, Krull-Gabriel dimension, Cantor-Bendixson rank and the relations between these are briefly recalled in Section~\ref{CBrank-KGdim}.

Section~\ref{structure} reviews what we need on (homotopy) string complexes and the morphisms between them.  In particular, we describe the structure of the Auslander--Reiten quiver of the string complexes and determine the Hom-hammocks of each string complex.

In Section~\ref{KGdim-simples} we consider the triangulated category $\KProjL$ where $\Lambda$ is a derived-discrete algebra.  We first show that the Krull-Gabriel dimension is defined, then determine the simple functors in each quotient category obtained from the Krull-Gabriel filtration, and then identify the indecomposable complex relatively isolated by each such functor.

By that point we have essentially complete information on the topology of the Ziegler spectrum of $\KProjL$ and the remainder of the paper is devoted to showing that there are no more indecomposable objects in $\KProjL$.

Our proof of that is structured around the Cantor-Bendixson analysis of the Ziegler spectrum.  After some preliminary results in Section~\ref{ind-compact-supp}, we treat the infinite global dimension case in Section~\ref{ind-infinite}, this being simplified by there being unbounded, CB-rank 1, compact objects.  We have to work harder in the finite global dimension case, in Section~\ref{ind-finite}, where an analysis of the structure of the indecomposables of CB-rank 1 allows us to complete the proof.

\begin{convention}
Throughout, $\Lambda$ will be a derived-discrete algebra over an algebraically closed field $\kk$, which is not piecewise hereditary of Dynkin type. The homotopy category of (not necessarily bounded) complexes of projective left $\Lambda$-modules will be denoted by the shorthand $\sK \coloneqq \KProjL$.
\end{convention}

\begin{acknowledgments}
The first author kindly acknowledges The University of Manchester for their hospitality on a research visit there where this work was begun.  The second author was supported by a School of Mathematics, University of Manchester, Scholarship.  The last two named authors acknowledge the financial support of the EPSRC through grant EP/K022490/1.
\end{acknowledgments}

\section{The Ziegler spectrum of a compactly generated triangulated category} \label{background}

Throughout this section $\T$ will be a compactly generated triangulated category with 
set-indexed coproducts; we shall denote the suspension functor by $\Sigma \colon \T \to \T$. For the convenience of the reader we briefly recall the definition here. A standard reference for this material is \cite{Neeman-book}.

An object $C \in \T$ is \emph{compact} if, for every set ${\{ D_i \mid i \in I\}}$ of objects in $\T$, the canonical morphism 
\[ 
\Hom_\T(C, \bigoplus_{i\in I} D_i) \to \bigoplus_{i\in I} \Hom_\T(C, D_i)
\] 
is an isomorphism.  We will denote the full (triangulated) subcategory of compact objects in $\T$ by $\Tc$. We say that $\T$ is \emph{compactly generated} if $\Tc$ is skeletally small and, for every object $0 \neq D$ in $\T$, there exists a nonzero morphism $C \to D$ for some $C\in \Tc$.

\subsection{Purity in a compactly generated triangulated category}

Notions of purity and the Ziegler spectrum were first developed in the context of categories of modules.  They were extended to compactly generated triangulated categories $\T$ by various authors, in particular in \cite{KraTel} (see also \cite{Belig} and Neeman's review \cite{NeeReview} of \cite{BenGnac} by Benson and Gnacadja). 
For an overview of purity in compactly generated triangulated categories we refer to \cite[Ch.~17]{PreNBK}.

Consider the category of contravariant functors from $\Tc$ to the category $\Ab$ of abelian groups:
\[
\Mod{\Tc} = ((\Tc)^{\op}, \Ab).  
\]

\begin{definition}
Let $M,N$ be objects of $\T$ and $f\colon M \to N$ a morphism in $\T$. The \emph{restricted Yoneda functor} $Y \colon \T \to \Mod{\Tc}$  is defined by 
\[ 
N \mapsto \Hom_\T(-,N)|_{\Tc}
\quad \text{and} \quad
Y(f) = \Hom_\T(-,f)|_{\Tc}.
\] 
\end{definition}

\begin{convention} \label{convention}
Unless otherwise specified, the shorthand $(-,N)$ stands for the restricted \emph{contravariant} functor $\Hom_{\sT}(-,N)|_{\Tc}$. However, the shorthand $(N,-)$ will be used for the \emph{unrestricted covariant} functor $\Hom_{\sT}(N,-)$, except briefly in Theorem~\ref{thm:Zg-homeomorphism} and Remark~\ref{rem: lim ext represent}. Similarly for the action on morphisms.
\end{convention}

A monomorphism $f \colon M \to N$ in $\T$ is a \emph{pure-monomorphism} if $(-,f)$ is a monomorphism in $\Mod{\Tc}$.  We will take the following equivalent statements as the definition of a \emph{pure-injective object} in $\T$.  

\begin{proposition}[{\cite[\S 1.4]{KraTel}}]\label{prop: pure injective}
Let $\T$ be a compactly generated triangulated category and let $N$ be an object of $\T$.  Then the following statements are equivalent: \begin{enumerate}
\item The functor $(-,N)$ is an injective object of $\Mod{\Tc}$.
\item For every $M \in \T$ the induced morphism $(M,N) \to \big( (-,M), (-,N) \big)$ is an isomorphism.
\item The object $N$ is injective over all pure-monomorphisms.  In other words, for every pure-monomorphism $g \colon L \to M$ and every morphism $f \colon L \to N$, there exists a morphism $h \colon M \to N$ such that the following diagram commutes:  \[ \xymatrix{ L \ar[r]^{g}_{pure} \ar[d]_f & M \ar@{.>}[dl]^h \\ N &
}\]
\end{enumerate}
\end{proposition}

It then follows (see \cite[Cor.~1.9]{KraTel}) that the restricted Yoneda functor induces an equivalence of categories 
\begin{equation*} \label{pinj-inj-equivalence}
\Pinj{\T} \simeq \Inj{\Mod{\Tc}}.
\end{equation*}
where $\Pinj{\T}$ is the full subcategory of pure-injective objects of $\T$ and $\Inj{\Mod{\Tc}}$ is the full subcategory of injective objects in $\Mod{\Tc}$.

\subsection{The Ziegler Spectrum of $\T$}

To define the compact open sets in the Ziegler spectrum, we require the notion of coherent functors from $\T$ to $\Ab$. A functor $F \colon \T \to \Ab$ is \emph{coherent} if there exists an exact sequence 
\[ (D, -) \to (C,-) \to F \to 0 \]
for some $C, D \in \Tc$.  The category of coherent functors from $\T$ to $\Ab$ will be denoted $\Coh{\T}$.

\begin{definition}\label{def:opens}
The \emph{Ziegler spectrum} $\Zg(\T)$ of $\T$ is a topological space whose points are the isomorphism classes of indecomposable pure-injective objects in $\T$.  For each $F \in \Coh{\T}$, we define the set \[ (F) := \{ N\in \Zg(\T) \mid F(N) \neq 0\}. \]  These sets form a basis of (compact) open sets.
\end{definition}

\begin{definition} \label{def:definable}
A full subcategory $\mcD$ of $\T$ is \emph{definable} if it is of the form \[ \mcD = \{ M \in \T \mid F_i(M) = 0 \} \] for some set $\{ F_i \mid i \in I \}$ of coherent functors.

For a collection of objects $\sC \subseteq \T$ the \emph{definable subcategory generated by $\sC$} is
\[
\Def{\sC} \coloneqq \{ M \in \T \mid F \in \Coh{\T} \text{ and } F(C) = 0 \text{ for all } C \in \sC \implies F(M) =0 \}.
\]
In particular, a subcategory $\mcX$ is definable if and only if $\Def{\mcX} = \mcX$.
\end{definition}

In \cite[Prop.~7.3]{KraStab} Krause shows that there is a natural bijection between the definable subcategories of $\T$ and the closed subsets of $\Zg(\T)$.  As a result, each closed subset of $\Zg(\T)$ has the form $\mcD \cap \Zg(\T)$ for some definable subcategory $\mcD$ of $\T$.  

\subsection{The Ziegler spectrum of $\Mod{\Tc}$} 

Since $\Mod{\Tc}$ is a locally finitely presented abelian category we can make use of the corresponding definitions of purity, Ziegler spectrum and definable subcategory.  Although $\T$ is not locally finitely presented, nor even necessarily finitely accessible, these definitions take exactly the same form for a compactly generated triangulated category as they do in those other cases; so the reader may obtain the corresponding definitions for $\Mod{\Tc}$ by replacing $\T$ by $\Mod{\Tc}$ and $\Tc$ by $\mod{\Tc}$, where the last is the full subcategory of finitely presented functors.

\begin{definition} \label{def: absolutely pure fp inj} 
A functor $F \in \Mod{\Tc}$ is \emph{absolutely pure} if every embedding $F \to G$ in $\Mod{\Tc}$ is a pure-monomorphism.  We say that $F$ is \emph{fp-injective} if it is injective over all embeddings $\tau \colon G \to H$ such that $\coker\tau$ is finitely presented.
\end{definition}

These two properties are well-known to be equivalent, see for example \cite[Prop.~2.3.1]{PreNBK}.  

\begin{proposition} \label{prop: absolutely pure fp inj}
For $F \in \Mod{\Tc}$, the following are equivalent \begin{enumerate}
\item $F$ is absolutely pure;
\item $F$ is fp-injective;
\item $\Ext^1(G, F) = 0$ for each $G \in \mod{\Tc}$.
\end{enumerate}
\end{proposition}

Definition~\ref{def: absolutely pure fp inj} and Proposition~\ref{prop: absolutely pure fp inj} work when $\Tc$ is replaced with any skeletally small preadditive category but since we are considering a triangulated category we have the following additional characterisation of absolutely pure objects, see \cite[Lem.~2.7]{KraTel}.

\begin{proposition} \label{prop:abseqflat}
A functor $F \in \Mod{\Tc}$ is absolutely pure if and only if it is flat. In particular, every absolutely pure functor $F$ is a direct limit of representable functors. 
\end{proposition}

We will denote the full subcategory of $\Mod{\Tc}$ consisting of the absolutely pure objects by $\Abs{\Tc}$.  Because $\Mod{\Tc}$ is locally coherent, this is a definable subcategory (see \cite[Thm.~3.4.24]{PreNBK}), so we let $\Zg(\Abs{\Tc})$ denote the closed subset $\Abs{\Tc} \cap \Zg(\Mod{\Tc})$ with the topology induced from $\Zg(\Mod{\Tc})$.

\subsection{The spaces $\Zg(\T)$ and $\Zg(\Abs{\Tc})$ are homeomorphic}

In order to prove the main theorem of this section, we will need the following result. 

\begin{lemma}[{\cite[Lem.~7.2]{KraStab}}] \label{lem:functors}
Let $\T$ be a compactly generated triangulated category. Then there is an equivalence of categories
\[
(\mod{\Tc})\op \rightiso \Coh{\T} \text{ given by } F \mapsto F^\vee ,
\]
which is defined on objects $M \in \T$ by $F^\vee (M) \coloneqq (F,(-,M))$.
\end{lemma}

\begin{theorem}\label{thm:Zg-homeomorphism}
The assignment $N \mapsto (-,N) $ defines a homeomorphism between topological spaces $\Zg(\T)$ and $\Zg(\Abs{\Tc})$.
\end{theorem}

\begin{proof}
The points of $\Zg(\Abs{\Tc})$ are exactly the (isomorphism classes of) indecomposable injective objects in $\Mod{\Tc}$.  It therefore follows from the comment after Proposition~\ref{prop: pure injective} that the assignment $N \mapsto (-,N)$ is a bijection between the points of $\Zg(\T)$ and $\Zg(\Abs{\Tc})$.  It remains to show that this bijection is a homeomorphism.

We consider the topology on $\Zg(\Abs{\Tc})$.  The compact open sets of $\Zg(\Abs{\Tc})$ are, as in Proposition~\ref{def:opens}, those of the form $(G)$ where $G$ is an object from a particular localisation of $(\mod{\Tc}, \Ab)^{\fp}$ which is a category of functors on $\Abs{\Tc}$ and which we will denote by $\fun{\Abs{\Tc}}$; see \cite[\S 12.3]{PreNBK}, especially \cite[Prop.~12.3.21]{PreNBK}, for details.  It is also the case that $\Abs{\Tc}$ is the category of exact functors on $\fun{\Abs{\Tc}}$ and that $\fun{\Abs{\Tc}}$ is determined by $\Abs{\Tc}$ up to natural equivalence, see \cite[\S 18.1.2]{PreNBK}.  So, by \cite[Prop.~7.2]{PreCatImag} we have that $\fun{\Abs{\Tc}} \simeq (\mod{\Tc})^{\op}$ and this equivalence allows us to describe the open sets of $\Zg(\Abs{\Tc})$ directly in terms of isomorphism classes of objects from $\mod{\Tc}$.  One can also check that the equivalence  $(\mod{\Tc})^{\op} \simeq \fun{\Abs{\Tc}}$ takes an object $A \in (\mod{\Tc})^{(\op)}$ to the image of the representable functor $(A,-) \in (\mod{\Tc}, \Ab)^{\fp}$ in the localisation $\fun{\Abs{\Tc}}$ (this follows from the comments before \cite[Prop.~7.1]{PreCatImag} and the fact that $(A,-)$ is the dual of the functor $A\otimes_{\Tc}-$). It follows that the compact open sets are those of the form \[ (A) = \{ (-,N) \in \Zg(\Abs{\Tc}) \mid (A, (-,N)) \neq 0 \}\] for $A \in \mod{\Tc}$.  

The equivalence in Lemma~\ref{lem:functors} gives us that, for each $F \in \Coh{\T}$, there exists some $A \in \mod{\Tc}$ such that $A^\vee =F$ and the corresponding compact open set in $\Zg(\T)$ is 
\[ (F) := \{ N \in \Zg(\T) \mid F(N) \neq 0 \} = \{ N \in \Zg(\T) \mid (A, (-,N)) \neq 0 \}.\]
Thus $N \mapsto (-,N)$ defines a homeomorphism and we complete the proof.
\end{proof}

\begin{remark}\label{rem: lim ext represent}
We recall how to define the unique extension, $\F$, of any functor $F \in (\mod{\Tc}, \Ab)^{\fp}$ to a functor from $\Mod{\Tc}$ to $\Ab$ that commutes with direct limits. Let $M \in \Mod{\Tc}$, then $M$ can be written as a direct limit $M = \colim M_\lambda$  with $M_\lambda \in \mod{\Tc}$.  Then  \[ \F(M) \coloneqq \colim F(M_\lambda).\] 

In order to define the sets $(A)$ as above, one should note that the extension of $(A,-)$ from a functor on $\mod{\Tc}$ to one on $\Mod{\Tc}$ is still, since $A$ is finitely presented, the functor $(A, -)$.
\end{remark}

Recall that a typical object $G$ in $\Coh{\T}$ has a presentation $(D,-) \to (C,-) \to G \to 0$ where $C,D \in \Tc$.  Since the Yoneda functor is full, there must be some $f \colon C \to D$ such that $G = \coker(f,-)$; then we will denote $G$ by $F_f$.  It will be useful to be able to explicitly describe the form of $F \in \mod{\Tc}$ when $F^\vee = F_f$.

\begin{corollary} \label{cor:functors-explicit}
Consider the equivalence of Lemma~\ref{lem:functors} above. Suppose $F^\vee = F_f$ for some $f\colon C \to D$ in $\Tc$. Then $F$ is the kernel of the image of $f$ under the restricted Yoneda functor, i.e. there is an exact sequence
\[
0 \too F \too (-,C) \rightlabel{(-,f)} (-,D).
\]
In particular, every finitely presented $\Tc$-module is the kernel of a morphism between finitely generated projective $\Tc$-modules.
\end{corollary}

\begin{proof}
Suppose, under the equivalence of Lemma~\ref{lem:functors}, $F^\vee = F_f$ for some $f \colon C \to D$ in $\Tc$.  Note that $(-,D)^\vee = (D, -)$ whenever $D$ is a compact object since $((-, D), (-,M)) \cong (D, M)$ for any $M$ in $\T$.  We can rewrite the exact sequence \[ (D,-) \rightlabel{(f,-)} (C,-) \too F_f \too 0 \: \hbox{  as  } \: (-,D)^\vee \rightlabel{(-,f)^\vee} (-,C)^\vee \too F^\vee \too 0\] and so $0 \too F \too (-, C) \rightlabel{(-,f)} (-,D)$ is also exact.
\end{proof}

\section{Cantor-Bendixson rank and Krull-Gabriel dimension} \label{CBrank-KGdim}

\subsection{Cantor-Bendixson rank of a topological space}

The following material is classical; see, for example, \cite[\S 5.3.6]{PreNBK}.
Let $T$ be a topological space, for instance $\Zg(\T)$. A point $p \in T$ is \emph{isolated} if $\{p\}$ is an open set and such points are assigned \emph{Cantor-Bendixson rank} (or \emph{CB-rank}) $0$. The first \emph{Cantor-Bendixson derivative} $T'$ of $T$ is the closed subset of $T$ containing all the non-isolated points of $T$. With respect to the relative topology, $T'$ itself may have isolated points, which are said to have CB rank $1$. One continues this process inductively to obtain points of CB rank $n$ for any $n \in \bN$ and we write $T^{(n)} = (T^{(n-1)})'$, with $T = T^{(0)}$. If at some point $T^{(n+1)} = \emptyset$ but $T^{(n)} \neq \emptyset$ then we say that $T$ has \emph{CB rank $n$}.

\begin{remark}
The process above can be continued transfinitely.  Since the topological spaces we consider turn out to have finite CB rank, we ignore this.
\end{remark}

\subsection{Localisation at Serre subcategories}

 A full subcategory $\sS$ of an abelian category $\sC$ is a \emph{Serre subcategory} if, for every exact sequence $0\to A \to B  \to C \to 0$ in $\sC$, we have $ A, C \in \sS$  if and only if $B \in \sS$.

For a Serre subcategory $\sS \subseteq \sC$, the quotient category $\sC / \sS$ has the same objects as $\sC$ and morphisms are given by
\[ 
\Hom_{\sC/\sS}(A, B) := \varinjlim_{A/A' \in \mcS, B'\in \mcS} \Hom_{\sC}(A', B/B').
\]

\subsection{Localisation with respect to hereditary torsion theories} \label{sec:tp-localisation}

For background on (hereditary) torsion theories and localisation with respect to torsion classes we refer the reader to \cite[Ch.~VI]{Stenstrom}; see also \cite[Ch.~11]{PreNBK}. We remind the reader of some key points that we will need here.

Let $\sC$ be a Grothendieck abelian category. A \emph{torsion class} $\sA$ is a full subcategory $\sC$ that is closed under quotients, extensions, direct sums and summands; objects of $\sA$ are said to be \emph{torsion}. A torsion class is \emph{hereditary} if it is also closed under subobjects. It determines a \emph{torsion pair} $(\sA,\sB)$, where $\sB = \sA\orth \coloneqq \{ B \in \sC \mid \Hom(\sA,B) = 0 \}$. An object of $\sB$ is said to be \emph{torsionfree}. A torsion class comes equipped with a subfunctor $t$ of the identity functor on $\sC$ which sends an object $C$ to $t(C)$, the largest torsion subobject of $C$. If the functor $t$ commutes with direct limits, $(\sA, \sB)$ is said to be of \emph{finite type}.

A hereditary torsion theory determines a \emph{localisation functor}, $q \colon \sC \to \sC/\sA$, which may be defined as in the subsection above or (on objects) by
\[
C \mapsto \pi^{-1}(t(E(C')/C')),
\]  
where $C' := C/t C$, $E(C')$ is the injective hull of $C'$ and $\pi = \coker(C' \into E(C'))$.
The localisation functor $q$ enjoys the following properties:
 \begin{enumerate}
\item $q$ is left adjoint to the canonical embedding $i \colon \sC/\sA \to \sC$; 
\item if $E$ is a torsionfree injective object then $E \cong i \circ q(E)$.
\end{enumerate}

Since $\Mod{\Tc}$ is locally coherent, $\mod{\Tc}$ is abelian and the $\varinjlim$-closure of each Serre subcategory of $\mod{\Tc}$ is a hereditary finite-type torsion class. 

\subsection{Localisation with respect to definable subcategories} \label{sec:def-localisation}

We refer to \cite[\S 12.3, \S 12.4]{PreNBK} for details of the following.    

We saw in the proof of Theorem~\ref{thm:Zg-homeomorphism} that $\fun{\Abs{\Tc}} \simeq (\mod{\Tc})^{\op}$.  It follows that each definable subcategory $\mcD$ of $\Abs{\Tc}$ is determined by the Serre subcategory
\[
\sS(\mcD) := \{ F \in \mod{\Tc}  \mid (F, D) =0 \text{ for all } D \in \mcD\}
\]
of $\mod{\Tc}$ which then generates a hereditary, finite-type, torsion subclass in $\Mod{\Tc}$
\[ 
\sA(\mcD) = \{A \in \Mod{\Tc} \mid (A,D) = 0 \hbox{ for all } D \in \mcD \}.
\]  
We shall denote the corresponding localisation of $\Mod{\Tc}$ by  $(\Mod{\Tc})_\mcD$.

By \cite[Prop.~11.1.29]{PreNBK}, the torsion theory $(\sA(\mcD), \sB(\mcD))$ is cogenerated by its indecomposable torsionfree injective objects, i.e.\ the set $\cl{X} := \mcD \cap \Zg(\Abs{\Tc})$, which is a closed subset of $\Zg(\Abs{\Tc})$.  By abuse of notation we shall often write $(\Mod{\Tc})_{\cl{X}}$ for the localisation $(\Mod{\Tc})_\mcD$.  

The category of finitely presented objects in $(\Mod{\Tc})_\mcD$ and the quotient of $\mod{\Tc}$ by $\sS(\mcD)$ coincide (\cite[Cor.~11.1.34]{PreNBK}).  Again, by abuse of notation we denote this localisation by $(\mod{\Tc})_\mcD$ or $(\mod{\Tc})_{\cl{X}}$.  Since $\Coh{\T} \simeq (\mod{\Tc})^{\op}$ by Lemma~\ref{lem:functors},  we use similar notation for the corresponding localisations of $\Coh{\T}$.

Given the homeomorphism $\Zg(\T) \cong \Zg(\Abs{\Tc})$ of Theorem~\ref{thm:Zg-homeomorphism}, there are parallel definitions and notations for localisations of $\Coh{\T}$ at definable subcategories $\mcD$ of $\T$ and closed subsets $\cl{X} \subseteq \Zg(\T)$.

\subsection{Krull-Gabriel dimension}

Recall from \cite[Def.~2.1]{Geigle} that a \emph{Krull-Gabriel filtration} of a (skeletally) small abelian category $\sC$ consists of a properly increasing filtration by Serre subcategories
\[
0 = \sC_{-1} \subseteq \sC_0 \subseteq \sC_1 \subseteq \cdots \subseteq \sC_n = \sC,
\]
where, for each $i$, $\sC_i/\sC_{i-1}$ is the full subcategory of all objects of finite length in $\sC/\sC_{i-1}$. If there is such a filtration then the \emph{Krull-Gabriel dimension} (or \emph{KG-dimension}), $\KG(\sC)$, of $\sC$ is, $n$. For each $i$, we denote the localisation functor by $q_i \colon \sC \to \sC/\sC_i$.

By abuse of notation, in the case of a compactly generated triangulated category $\T$ we shall write $\KG(\T)$ for $\KG(\Coh{\T})$ and call it the \emph{KG-dimension of $\T$}.

\begin{remark}
More generally, ordinal-indexed Krull-Gabriel filtrations may be considered giving, as with CB-rank, an ordinal-valued KG-dimension.  In the case that there is no such filtration, the KG-dimension is undefined. 
\end{remark}

Let $L$ be a modular lattice and let $L'$ be the quotient modular lattice obtained by collapsing all finite length intervals in $L$. Starting with $L_{-1} \coloneqq L$ and defining $L_n = (L_{n-1})'$ for $n \in \bN$, we define the \emph{m-dimension} of $L$ to be the least $n$ such that $L_n = 0$.  We recall the following lemma, see for example \cite[Lem.~1.1]{KraGen}, where $L_\sC(X)$ denotes the lattice of subobjects of an object $X$ of an abelian category $\sC$.

\begin{lemma} \label{lem: m-dimension}
Let $\sC$ be an abelian category and $X \in \sC$ an object. Then $q_i$ induces a natural isomorphism $L_\sC(X)_i \cong L_{\sC/\sC_i}(q_iX)$ for all $i$.
\end{lemma}

\begin{remark} \label{KGdimmdim}
We highlight the following consequences of Lemma~\ref{lem: m-dimension}.
\begin{enumerate}
\item For $X\neq 0$ the m-dimension of $L_{\sC}(X)$ is $n$ if and only if $q_{n-1}(X) \neq 0$ and has finite length in $\sC/\sC_n$.
\item The Krull-Gabriel dimension of $\sC$ is defined if and only if for each $X\in \sC$ no subset of the lattice $L_{\sC}(X)$ forms a densely ordered chain (see, for example, \cite[Prop.~7.2.3, Prop.~13.2.1]{PreNBK}).
\end{enumerate}
\end{remark}

\begin{definition} \label{def:isolation-condition}
The Ziegler spectrum of $\T$ satisfies the \emph{isolation condition} if for every closed subset $\cl{X} \subseteq \Zg(\T)$ and every point $M \in \cl{X}$ isolated in the relative topology on $\cl{X}$, there is a coherent functor $F$ such that $(F) \cap \cl{X} = \{M\}$ and such that the image of $F$ is simple in the localisation $\Coh{\T}_{\cl{X}}$.
\end{definition}

We have the following link between KG-dimension and the isolation condition, which we follow with the link between the CB-rank and KG-dimension.

\begin{lemma}[{\cite[Lem.~8.11]{Ziegler}}] \label{lem:isolation}
For a compactly generated triangulated category $\T$, if ${\KG(\T)}$ is defined, then $\Zg(\T)$ satisfies the isolation condition.
\end{lemma}

\begin{proposition}[{\cite[Thm.~8.6]{Ziegler}, see also \cite[Prop.~10.19]{PrestMTM}}] \label{prop:KG-CB}
Suppose $\Zg(\T)$ satisfies the isolation condition. Then:
\begin{enumerate}
\item  $\Coh{\T}/\Coh{\T}_i = \Coh{\T}_{\cl{X}_i}$, where $\cl{X}_i$ is the closed subset of $\Zg(\T)$ containing the points of $\CB \geq i$;
\item there is a bijection between points $N$ of CB-rank $i$ and simple functors $F$ in $\Coh{\T}_{\cl{X}_i}$ given by $(F')\cap \cl{X}_i = \{ N\}$ where $F'\in \Coh{\T}$ is such that $q_iF'=F$.
\end{enumerate}
\end{proposition}

Here we are using the homeomorphism of Theorem~\ref{thm:Zg-homeomorphism} and the fact that the original result, proved for modules over rings, holds equally true for modules over small preadditive categories (the development of the theory in \cite{PreNBK} accommodates this level of generality).

\begin{corollary} \label{cor:CB-rank}
Suppose $\T$ is a compactly generated triangulated category. For an object $M$ of $\T$ we have 
\[
\CB M \text{ in } \Zg(\T) = \CB (-,M) \text{ in } \Zg(\Abs{\Tc}).
\]  If $\Zg(\T)$ satisfies the isolation condition then $\KG(\T) = \CB\Zg(\T)$. 
\end{corollary}

\begin{proof}
This follows immediately from Theorem~\ref{thm:Zg-homeomorphism} and Proposition~\ref{prop:KG-CB}.
\end{proof}

\section{The structure of $\KProjL$ for derived-discrete algebras} \label{structure}

Let $\Lambda$ be a derived-discrete algebra. In \cite[Thm.~A]{BGS}, Bobi\'nski, Gei\ss\ and Skowro\'nski obtained a canonical representative $\LLambda$, with $1 \leq r \leq n$ and $m \geq 0$, given by the bound quiver below, for each derived-equivalence class of derived-discrete algebra.
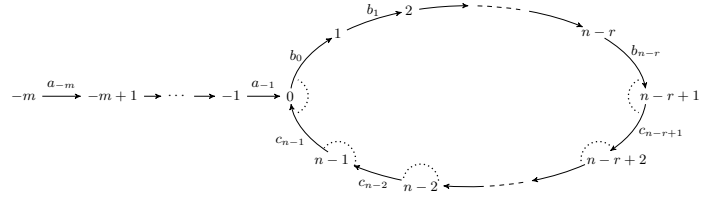
\[
\scalebox{0.6}{
\begin{tikzpicture}[xscale=1.3]
  \node (-m) at (-4.5,0) [tinyvertex] {$-m$};
  \node (-m+1) at (-3,0) [tinyvertex] {$-m+1$};
  \node (-m+2) at (-1.9,0) [tinyvertex] {$\cdots$};
  \node (-1) at (-1,0) [tinyvertex] {$-1$};
  \node (0) at (0,0) [tinyvertex] {$0$};
  \draw[->] (-m) to node[tinyvertex, above] {$a_{-m}$} (-m+1);
  \draw[->] (-m+1) to (-m+2);
  \draw[->] (-m+2) to (-1);
  \draw[->] (-1) to node[tinyvertex, above] {$a_{-1}$}(0);
  \node (1) at (0.8,1.4) [tinyvertex] {$1$};
  \node (n-1) at (0.7,-1.4) [tinyvertex] {$n-1$};
  \node (2) at (2,1.9) [tinyvertex] {$2$};
  \node (n-2) at (2.2,-2) [tinyvertex] {$n-2$};
  \node (n-r) at (5.2,1.4) [tinyvertex] {$n-r$};
  \node (n-r+2) at (5.5,-1.4) [tinyvertex] {$n-r+2$};
  \node (n-r+1) at (6.4,0) [tinyvertex] {$n-r+1$};
  \draw[->] (3,0) [partial ellipse=40:5:3cm and 2cm]; 
  \draw[->] (3,0) [partial ellipse=-5:-37:3cm and 2cm]; 
  \draw[->] (3,0) [partial ellipse=175:140:3cm and 2cm]; 
  \draw[->] (3,0) [partial ellipse=133:112:3cm and 2cm]; 
  \draw[->] (3,0) [partial ellipse=218:185:3cm and 2cm]; 
  \draw[->] (3,0) [partial ellipse=248:230:3cm and 2cm]; 
  \draw[->] (3,0) [partial ellipse=106:90:3cm and 2cm];
  \draw[dashed] (3,0) [partial ellipse=87:73:3cm and 2cm];
  \draw[->] (3,0) [partial ellipse=70:50:3cm and 2cm];
  \draw[->] (3,0) [partial ellipse=-49:-69:3cm and 2cm];
  \draw[dashed] (3,0) [partial ellipse=-72:-84:3cm and 2cm];
  \draw[->] (3,0) [partial ellipse=-85:-98:3cm and 2cm];
  \node at (0.1,0.9) [tinyvertex] {$b_0$};
  \node at (1.4,1.9) [tinyvertex] {$b_1$};
  \node at (6,1) [tinyvertex] {$b_{n-r}$};
  \node at (6.25,-0.8) [tinyvertex] {$c_{n-r+1}$};
  \node at (0,-1) [tinyvertex] {$c_{n-1}$};
  \node at (1.4,-1.95) [tinyvertex] {$c_{n-2}$};
  \draw (0) [dotted, thick, partial ellipse=-65:65:0.3cm and 0.4cm];
  \draw (6,0) [dotted, thick, partial ellipse=115:245:0.3cm and 0.4cm];
  \draw (5.2,-1.4) [dotted, thick, partial ellipse=52:195:0.3cm and 0.4cm];
  \draw (0.8,-1.4) [dotted, thick, partial ellipse=-20:140:0.3cm and 0.4cm];
  \draw (2.2,-1.9) [dotted, thick, partial ellipse=5:170:0.3cm and 0.4cm];
\end{tikzpicture}} \label{page:quiver}
\]
\noindent
In \cite[Thm.~B]{BGS}, the authors then used the triple $(r,n,m)$ to describe the structure of the Auslander--Reiten (AR) quiver of $\Db(\mod{\LLambda})$. Thus, from now on, we shall identify each derived-discrete algebra, $\Lambda$, with its representative $\LLambda$ in the derived equivalence class and use this algebra to describe the structure of $\KProjL$.

\subsection{String complexes}

Since $\Lambda$ is a gentle algebra the indecomposable complexes of $\Db(\mod{\Lambda}) \simeq \KprojL$ can be described in terms of so-called \emph{(homotopy) string complexes}; \cite{BM} (also \cite{BD}), and \cite{Bo} for the terminology. In this section we describe a generalisation to  string complexes with possibly unbounded (but still degreewise finite-dimensional) cohomology. We refer to \cite[\S 2]{ALP} for a summary of \cite{BM} using the terminology and notation of this article. 

\begin{remark}
The main theorem of \cite{BM} asserts that the indecomposable complexes of $\Db(\mod{\Lambda})$ are precisely the isoclasses of string and band complexes. However, the absence of band complexes is precisely the reason for discreteness, and we therefore make no further reference to band complexes.
\end{remark}

We begin with a description of the string complexes in $\KProjL$, which is an extension of \cite[Lem.~7.1]{ALP}; see also \cite[Lem.~3]{HanZhang}.
The proof of the following lemma is a straightforward exercise in the combinatorics of the quiver of $\LLambda$.

\begin{lemma} \label{lem:strings}
Each homotopy string of $\LLambda$ is a (shifted) copy of a subword of $w_\ell$ for $\ell \geq 0$ and $w_\infty$, ${}_\infty w$ and ${}_\infty w_\infty$ or its inverse:
\[ 
\xymatrix@C=2.8pc{
w_\ell \colon 
  & \bullet 
  & {\circ} \ar[l]_{a_{-1}\ldots a_{-m}} \ar[r]^-{v_k} 
  &   \circ \ar[r]^{c_{n-1}} 
  &  \cdots \ar[r]^{c_{n-r+1}}  
  & \bullet \ar[r]_{b_{n-r}\ldots b_0 a_{-1}\ldots a_{-m}} 
  & \bullet
  &
\\
w_\infty \colon
  & \bullet 
  & {\circ} \ar[l]_{a_{-1}\ldots a_{-m}} \ar[r]^-{v} 
  &   \circ \ar[r]^-{v} 
  &  \circ \ar[r]^-{v}  
  & \circ \ar@{.>}[r] 
  & 
  &
\\
{}_\infty w \colon
  & \ar@{.>}[r] 
  & \circ \ar[r]^-{v} 
  & \circ \ar[r]^-{v}
  &   \circ \ar[r]^{c_{n-1}} 
  &  \cdots \ar[r]^{c_{n-r+1}}  
  & \bullet \ar[r]_{b_{n-r}\ldots b_0 a_{-1}\ldots a_{-m}} 
  & \bullet
\\
{}_\infty w_\infty \colon
  & \ar@{.>}[r]
  & \circ \ar[r]^-{v}
  & \circ \ar[r]^-{v}
  & \circ \ar[r]^-{v}
  & \circ \ar@{.>}[r]
  & ,
  &
}
\] 
where $v_\ell$ is the $\ell$-fold concatenation of  $\xymatrix@C=2.8pc{ v \colon \bullet
  \ar[r]^-{c_{n-1}} & \cdots \ar[r]^-{c_{n-r+1}} & \bullet
  \ar[r]^-{b_{n-r}\ldots b_0} & \bullet }$.
Note that for $\iLLambda$ we have $b_0 = c_0$ in our labelling convention. 
\end{lemma}

The passage from a homotopy string $w$ to the corresponding string complex $P_w$ is first described in \cite{BM,Bo}.  We follow the notation in \cite[\S 2]{ALP}.

\begin{definition} \label{string-complexes}
Let $P_w$ be a string complex for $\Lambda = \LLambda$. We say that $P_w$ is
\begin{description}[font=\normalfont]
\item[\emph{perfect}] if it is any shift of a string complex of a subword of $w_\ell$ for $\ell \geq 0$;
\item[\emph{left-infinite}] if it is any shift of a string complex of a left-infinite subword of ${}_\infty w$;
\item[\emph{right-infinite}] if it is any shift of a string complex of a right-infinite subword of $w_\infty$;
\item[\emph{one-sided}] if it is left- or right-infinite;
\item[\emph{two-sided}] if it is any shift of the string complex ${}_\infty w_\infty$.
\end{description}
We shall denote the set of all string complexes of $\Lambda$ by $\strings{\Lambda}$. 
\end{definition}

We have the following global statement regarding the dimensions of Hom-spaces between string complexes; cf. \cite[Thm~.5.1]{BPP1} and \cite[Thm.~7.4]{ALP}.

\begin{proposition} \label{dimensions}
Let $\Lambda$ be a derived-discrete algebra. Then, for $A, B \in \strings{\Lambda}$, we have 
\[
\dim \Hom_{\sK}(A,B) \leq 
\left\{ 
\begin{array}{ll} 
1 & \text{ for } r > 1, \\
2 & \text{ for } r = 1.
\end{array}
\right.
\]
Moreover, if one of $A, B \in \strings{\Lambda}$ is nonperfect then $\dim \Hom_{\sK}(A,B) \leq 1$.
\end{proposition}

\begin{proof}
For this we apply \cite[Thm.~3.15]{ALP} using arguments analogous to those in \cite[\S 7]{ALP}.
\end{proof}

To see which string complexes admit 2-dimensional Hom-spaces between them, we refer to \cite[Prop.~5.2]{BPP1}.

\begin{corollary} \label{strings-are-indecomposable}
Let $\Lambda$ be a derived-discrete algebra. Then the objects of $\strings{\Lambda}$ are indecomposable.
\end{corollary}

\begin{proof}
The perfect string complexes are known to be indecomposable by \cite{BM}. For any other string complex $A$ we have $\Hom_{\sK}(A,A) = \kk$, whence the idempotent-completeness of $\sK$ gives the indecomposability. 
\end{proof}

\subsection{Compact objects of $\sK$}

To show that string complexes are pure-injective, we first need to identify the compact objects of $\sK$. This is a special case of a construction of J\o rgensen \cite{Jorgensen}; see also \cite{Neeman} for a strengthening of this result.

\begin{proposition} \label{compact}
Let $\Lambda = \LLambda$ be a derived-discrete algebra. An indecomposable complex in $\sK$ is compact if and only if it is a right-infinite string complex.
\end{proposition}

\begin{proof}
We follow J\o rgensen's construction from \cite{Jorgensen}.
Write $(-)^* = \Hom_{\Lambda}(-,\Lambda)$ for the standard duality with respect to $\Lambda$. If $M$ is a finite-dimensional left $\Lambda$-module then $M^*$ is a finite-dimensional right $\Lambda$-module. Let $P_M$ be a projective resolution of $M^*$. Note that $P_M \in \Kminus{\Lambda\op}$ and is quasi-isomorphic to $M^*$ in $\Db(\mod{\Lambda\op})$. We also consider the complex $P_M^* \in \Kminus{\Lambda}$, and note that both $P_M$ and $P^*_M$ depend functorially on $M$ in $\Kminus{\Lambda\op}$ and $\Kminus{\Lambda}$, respectively; see \cite[Construction 1.2]{Jorgensen}.

Let
$\sG \coloneqq \{ \Sigma^i P^*_M \mid M \in \mod{\Lambda}, i\in \bZ \}$.
\cite[Thm.~2.4]{Jorgensen} asserts that $\sG$ is a set of compact generators for $\sK$. Now consider the following two subcategories of $\sK$ and $\sK(\Proj{\Lambda\op})$, where the notation indicates the thick subcategory generated by a set of objects:
\[
\sC \coloneqq \thick_{\sK}(\sG) \quad \text{and} \quad \sD \coloneqq \thick_{\sK(\Proj{\Lambda\op})}(G^* \mid G \in \sG).
\]
Note that $\sC = \sK^c$, since $\sG$ is a set of compact generators \cite[Thm.~2.1.3]{Neeman-Groth}. Now, the proof of \cite[Thm.~3.2]{Jorgensen} asserts that 
\[
\xymatrix{ \sC \ar@<1ex>[r]^-{(-)^*} & \ar@<1ex>[l]^-{(-)^*} \sD } 
\]
are quasi-inverse equivalences between triangulated categories.
In the proof of \cite[Thm.~3.2]{Jorgensen} it is shown that 
\[
\sD = \thick_{\sK(\Proj{\Lambda\op})}(P_N \mid P_N \text{ is a projective resolution of some } N \in \mod{\Lambda\op}).
\]
In particular, this means that $\sD = \Kminus{\Lambda\op}$, the indecomposable complexes of which are precisely the left-infinite string complexes in $\strings{\Lambda\op}$.
For $\Lambda$ derived-discrete, the image of the set of left-infinite string complexes in $\strings{\Lambda\op}$ under the equivalence $(-)^*$ is precisely the set of right-infinite string complexes in $\strings{\Lambda}$. Thus the right-infinite string complexes are precisely the indecomposable compact complexes of $\sK$.
\end{proof}

Recall from \cite{CBLoc,KraDec} that an object $N$ of a general compactly generated triangulated category is \emph{endofinite} if for each $C \in \Tc$ the $\End_\T(N)$-module $\Hom_\T(C,N)$ is of finite length.

\begin{lemma}[{\cite[Thm.~1.2]{KraDec}}] \label{endofinite}
An endofinite object of a compactly generated triangulated category is pure-injective.
\end{lemma}

Putting Propositions~\ref{dimensions}, \ref{compact} and Corollary~\ref{strings-are-indecomposable} together with Lemma~\ref{endofinite} gives:

\begin{corollary} \label{cor:pure-injective}
Let $\Lambda$ be a derived-discrete algebra. Then any string complex is endofinite and hence is an indecomposable pure-injective object of $\sK$.
\end{corollary}

We now know that the objects of $\strings{\Lambda}$ are indecomposable pure-injective objects of $\sK$. In Section~\ref{KGdim-simples} we shall see that $\strings{\Lambda}$ is a complete list of indecomposable pure-injective objects of $\sK$. However, in order to do this we need to describe the `AR quiver' of the additive category $\add{\strings{\Lambda}}$ and the shapes of the Hom-hammocks, which we do in the next subsections.

\subsection{The AR quiver of $\strings{\Lambda}$ for $\Lambda$ of finite global dimension} \label{sec:finite}

The AR quiver of $\strings{\Lambda}$ is computed by a straightforward calculation using Bob\'inski's algorithm \cite{Bo} (also, \cite[\S 6]{ALP} avoiding the Happel functor and including pictures).
It has $8r$ components: $3r$ components sit in $\KbprojL$, for which we keep the notation and co-ordinate system of \cite{BGS}. In particular, the components are denoted $\mcX^0, \ldots, \mcX^{r-1}$, $\mcY^0, \ldots, \mcY^{r-1}$, and $\mcZ^0,\ldots,\mcZ^{r-1}$. The components $\mcX^k$ and $\mcY^k$ are of type $\bZ A_\infty$. The components $\mcZ^k$ are of type $\bZ A_{\infty}^{\infty}$.
For each $k = 0,\ldots,r-1$, we label the indecomposable objects of $\mcX^k$, $\mcY^k$ and $\mcZ^k$ as follows:
\[
X^k_{ij}  \in \mcX^k \text{ with } i,j \in \bZ, j \geq i; \quad
Y^k_{ij}  \in \mcY^k \text{ with } i,j \in \bZ, j \leq i; \quad
Z^k_{ij}  \in \mcZ^k \text{ with } i,j \in \bZ.
\]

The remaining $5r$ components consist of $4r$ components of type $A_\infty^\infty$ and $r$ components of type $A_1$. The components of type $A_\infty^\infty$ will be denoted by $\mcX^k_\infty$, $\mcX^k_{-\infty}$,  $\mcY^k_\infty$ and  $\mcY^k_{-\infty}$, and the components of type $A_1$ by $\mcZ^k_\infty$. For each $k =0,\ldots,r$ we label the indecomposable objects of these components as follows:
\begin{align*}
X^k_{i,\infty} \in \mcX^k_\infty \text{ with } i \in \bZ, & \quad
X^k_{-\infty,i} \in \mcX^k_{-\infty} \text{ with } i \in \bZ, \quad  
Z^k_{\infty} \in \mcZ^k_\infty \\
Y^k_{\infty,i} \in \mcY^k_\infty \text{ with } i \in \bZ, & \quad
Y^k_{i,-\infty} \in \mcY^k_{-\infty} \text{ with } i \in \bZ.  
\end{align*}

For $a \in \bZ \cup \{-\infty\}$, $b \in \bZ \cup \{\infty\}$ and $0 \leq k < r-1$, suspension acts as follows.

\smallskip

\begin{tabular}{llll}
 $\Sigma X^k_{a,b} = X^{k+1}_{a,b} $                & $\Sigma Y^k_{a,b} = Y^{k+1}_{a,b}$            & $\Sigma Z^k_{a,b} = Z^{k+1}_{a,b} $              & $\Sigma Z^k_\infty = Z^{k+1}_\infty$  \\
 $\Sigma X^{r-1}_{a,b} = X^0_{a+r+m,b+r+m}$  & $\Sigma Y^{r-1}_{a,b} = Y^0_{a+r-n,b+r-n}$ & $\Sigma Z^{r-1}_{a,b} = Z^0_{a+r+m,b+r-n}$ &
\end{tabular}

\smallskip

\begin{remark}
Writing $\mcX \coloneqq \add{\bigcup_{k=0}^{r-1}\mcX^k}$ etc., we have the following correspondences:
\begin{eqnarray*}
\ind{\mcX \cup \mcY \cup \mcZ}       & \bij & \{\text{perfect string complexes}\}; \\
\ind{\mcX_\infty \cup \mcY_{-\infty})} & \bij & \{\text{right-infinite string complexes}\}; \\
\ind{\mcX_{-\infty} \cup \mcY_\infty)} & \bij & \{\text{left-infinite string complexes}\}; \\
\ind{\mcZ_\infty}                     & \bij & \{\text{two-sided string complexes}\}.
\end{eqnarray*}
Moreover, in light of Proposition~\ref{compact}, when $\gldim \Lambda < \infty$ the compact indecomposable complexes in $\sK$ are precisely those lying in $\mcX$, $\mcY$ and $\mcZ$ components of $\strings{\Lambda}$.
\end{remark}

\begin{figure}
\begin{center}
\begin{tikzpicture}[scale=0.1825,thick]
\newcommand{\sst}[1]{\scalebox{0.6}{$#1$}}

\draw[fill,blue!30] (18.3,-4.7) -- (17.7,-5.3) -- (22.7, -10.3) -- (23.3, -9.7) -- cycle;
\draw[fill,blue!30] (19,-4) -- (24,-9) -- (33,0) -- (28,5) -- cycle; 
\draw[fill,blue!30] (28.7,5.7) -- (29.3,6.3) -- (34.3,1.3) -- (33.7,0.7) -- cycle;
\draw[fill,blue!30] (30,7) -- (34,11) -- (39,6) -- (35,2) -- cycle;
\draw[fill] (20,-5) node [right]  {\sst{X^0_{a,\infty}}};
\fill (18,-5) circle (0.3);

\draw[fill,red!30] (35.7,1.3) -- (36.3,0.7) -- (45.3,9.7) -- (44.7,10.3) -- cycle;
\draw[fill,red!30] (35.7,-1.3) -- (36.3,-0.7) -- (45.3,-9.7) -- (44.7,-10.3) -- cycle;
\draw[fill,red!30] (37, 0) -- (46, 9) -- (55, 0) -- (46,-9) -- cycle;
\fill[red!30] (35,0) circle (0.7);

\draw[fill,green!30] (51.7,-4.7) -- (52.3,-5.3) -- (56.3,-1.3) -- (55.7,-0.7) -- cycle;
\fill[green!30] (57,0) circle (0.7);
\draw[fill,green!30] (57.7,1.3) -- (58.3,0.7) -- (67.3,9.7) -- (66.7,10.3) -- cycle;
\draw[fill,green!30] (53,-6) -- (58,-11) -- (62,-7) -- (57,-2) -- cycle;
\draw[fill,green!30] (57.7,-1.3) -- (58.3,-0.7) -- (63.3,-5.7) -- (62.7,-6.3) -- cycle;
\draw[fill,green!30] (59,0) -- (68,9) -- (73,4) -- (64,-5) -- cycle;
\draw[fill] (54, -5) node [right]  {\sst{X^2_{-\infty,b}}};
 \fill (52,-5) circle (0.3);

  \node at ( 7, 10)  {\sst{\mcY^0}};  \node at (29, 10)  {\sst{\mcX^1}};  \node at (51, 10)  {\sst{\mcY^2}};
  \node at ( 7,-10)  {\sst{\mcX^0}};  \node at (29,-10)  {\sst{\mcY^1}};  \node at (51,-10)  {\sst{\mcX^2}};
  \node at (17,0.2) {\sst{\mcZ^0}};  \node at (39,0.2) {\sst{\mcZ^1}};  \node at (61,0.2) {\sst{\mcZ^2}};

  \draw (0,9) -- (1,10);
  \draw (0,-9) -- (1,-10);

  \draw[dashed] (0,  7) -- ( 2, 9) -- (11, 0) -- (2,-9) -- (0,-7); 

  \draw (3,10) -- (12,1);
  \draw (3,-10) -- (12,-1);

  \draw[dashed] (4, 11) -- (13,  2) -- (22, 11);
  \draw         (4, 11) -- (22, 11);

  \draw (14, 1) -- (23,10);
  \fill (13,0) circle (0.4);
  \draw (14,-1) -- (23,-10);

  \draw[dashed] (4,-11) -- (13, -2) -- (22,-11);
  \draw         (4,-11) -- (22,-11);

  \draw[dashed] (15, 0) -- (24, 9) -- (33, 0) -- (24,-9) -- cycle;

  \draw (25, 10) -- (34,1);
  \draw (25,-10) -- (34,-1);

  \draw[dashed] (26, 11) -- (35,  2) -- (44, 11);
  \draw         (26, 11) -- (44, 11);

  \draw[dashed] (26,-11) -- (35, -2) -- (44,-11);
  \draw         (26,-11) -- (44,-11);

  \draw (36, 1) -- (45,10);
  \fill (35,0) circle (0.4);
  \draw (36,-1) -- (45,-10);

  \draw[dashed] (37, 0) -- (46, 9) -- (55, 0) -- (46,-9) -- cycle;

  \draw (47, 10) -- (56,1);
  \draw (47,-10) -- (56,-1);

  \draw[dashed] (48, 11) -- (57,  2) -- (66, 11);
  \draw         (48, 11) -- (66, 11);

  \draw[dashed] (48,-11) -- (57, -2) -- (66,-11);
  \draw         (48,-11) -- (66,-11);

  \draw (58, 1) -- (67,10);
  \fill (57,0) circle (0.4);
  \draw (58,-1) -- (67,-10);

  \draw[dashed] (59, 0) -- (68, 9) -- (77, 0) -- (68,-9) -- cycle;

  \draw (69, 10) -- (78,1);
  \draw (69,-10) -- (78,-1);
  \fill (79,0) circle (0.4);

  \draw[dashed] (70, 11) -- (79,  2) -- (81, 4);
  \draw         (70, 11) -- (81, 11);
 
  \draw[dashed] (70,-11) -- (79, -2) -- (81,-4);
  \draw         (70,-11) -- (81,-11);

  \draw (80,1) -- (81,2);
  \draw (80,-1) -- (81,-2);

\end{tikzpicture}
\end{center}
\caption{The structure of $\strings{\Lambda}$ when $\gldim \Lambda < \infty$. Triangular regions denote components of type $\bZ A_\infty$ and diamond shaped regions those of type $\bZ A_{\infty}^{\infty}$. Solid horizontal lines indicate boundaries or `mouths' of the components. 
The solid line to the left of an $\mcX^k$ component is an $\mcX^k_{-\infty}$ component and to the right an $\mcX^k_\infty$ component; similarly for $\mcY^k$, $\mcY^k_{-\infty}$ and $\mcY^k_\infty$. The $\mcZ^k_\infty$ components are indicated by the central dots to the left of the $\mcZ^k$ components.
The shaded regions illustrate the Hom-hammocks for the various indecomposable objects indicated; cf. \cite[\S 2]{BPP1}.\label{fig:finite}}
\end{figure}

\subsection{The AR quiver of $\strings{\Lambda}$ for $\Lambda$ of infinite global dimension} \label{sec:infinite}

The AR quiver of $\Db(\mod{\Lambda}) \simeq \KprojL \subset \strings{\Lambda}$ has $2r$ components consisting of $r$ components of type $\bZ A_\infty$ denoted by $\mcX^0,\ldots,\mcX^{r-1}$ and $r$ components of type $A^{\infty}_{\infty}$ denoted by $\mcZ^0, \ldots, \mcZ^{r-1}$. The full subcategory of perfect complexes $\KbprojL = \add{\mcX}$. We sketch the structure of $\KprojL$ in Figure~\ref{fig:left-infinite}; the co-ordinate system is that of \cite{BGS}. For each $k = 0,\ldots,r=n$, we label the indecomposable objects of $\mcX^k$ and $\mcZ^k$ as follows:
\[
X^k_{ij}  \in \mcX^k \text{ with } i,j \in \bZ, j \geq i; \quad
Z^k_{i}  \in \mcZ^k \text{ with } i \in \bZ.
\]

\begin{figure}
\begin{center}
\begin{tikzpicture}[scale=0.2,thick]

\newcommand{\sst}[1]{\scalebox{0.5}{$#1$}}

  \draw[fill,blue!30] (10,-7) -- (13,-10) -- (16, -7) -- (13, -4);
  \draw[fill,blue!30] (13.7,-3.3) -- (16.7,-6.3) -- (17.3,-5.7) -- (14.3,-2.7); 
  \draw[fill,blue!30] (15,-2) -- (18,1) -- (21,-2) -- (18,-5);
  \draw[fill] (10,-7) node [left =-0.1]  {$\scriptstyle X$}      circle (0.2);
  \draw[fill] (21, -2) node [right=-0.05] {$\scriptstyle \sS X$} circle (0.2);

  \draw[fill,red!30] (25.7,-4.7) -- (26.3,-5.3) -- (31.3,-0.3) -- (30.7,0.3);
  \draw[fill,red!30] (27,-6) -- (31,-10) -- (36,-5) -- (32,-1); 
  \draw[fill] (26,-5) node [left =-0.1] {$\scriptstyle Z$} circle (0.2);

  \node at (4,-9)  {\sst{\mcX^0}};  \node at (16,0)  {\sst{\mcX^1}};  \node at (27,-9)  {\sst{\mcX^2}}; \node at (38,0)  {\sst{\mcX^3}}; \node at (49,-9)  {\sst{\mcX^4}}; \node at (4,0)  {\sst{\mcX^{r-1}}};
  \node at (-0.7,-10) {\sst{\mcZ^{r-1}}};  \node at (10,1) {\sst{\mcZ^0}};  \node at (21,-10) {\sst{\mcZ^1}}; \node at (32,1) {\sst{\mcZ^2}}; \node at (43,-10) {\sst{\mcZ^3}}; \node at (54,1) {\sst{\mcZ^4}};

 \draw[dashed] (-1,-8) -- (8,1);
 \draw (-1,1) -- (8,1);
  \draw (0,-9) -- (9, 0); 
  \draw[dashed] (1,-10) -- (10, -1) -- (19,-10);
  \draw         (1,-10) -- (19,-10);
   \draw (11, 0) -- (20,-9);
  \draw (12,1) -- (30,1);
  \draw[dashed] (12,1) -- (21,-8) -- (30,1);
  \draw (22,-9) -- (31,0);
  \draw[dashed] (23,-10) -- (32, -1) -- (41,-10);
  \draw         (23,-10) -- (41,-10);
  \draw (33,0) -- (42,-9);
  \draw[dashed] (34, 1) -- (43,-8) -- (52, 1);
  \draw         (34, 1) -- (52, 1);
  \draw (44,-9) -- (53,0);
  \draw[dashed] (45,-10) -- (54, -1) -- (63,-10);
  \draw         (45,-10) -- (63,-10);
  \draw (55,0) -- (64,-9);
  \draw[dashed] (56,1) -- (65,-8);
  \draw (56,1) -- (65,1);
\end{tikzpicture}
\end{center}
\caption{The structure of $\KprojL$ when $\gldim \Lambda = \infty$. Notation for components as in Figure~\ref{fig:finite}.}\label{fig:left-infinite}
\end{figure}

The AR quiver of $\strings{\Lambda}$ has $3r$ components. The $r$ components $\mcX^k$ of $\Db(\mod{\Lambda})$ of type $\bZ A_\infty$ remain unchanged. The $r$ components $\mcZ^k$ of type $A^{\infty}_{\infty}$ now form `beams' of $r$ `ladder-type' components $\mctZ^k$:
\[
\xymatrix@!=7pt{
             &                        &                        &                             &                     & \\
             &                        &                        & X^k_{i+1,\infty} \ar@{..}[ur] \ar[dr] &                     & \\
             &                        & X^k_{i,\infty} \ar[ur] \ar[dr] &                             & Z^k_{i+1} \ar@{..}[ur] & \\  
             & X^k_{i-1,\infty} \ar[ur] \ar[dr] &                        & Z^k_i \ar[ur]              &                     & \\
\ar@{..}[ur] &                        & Z^k_{i-1} \ar[ur]         &                             &                     & \\
             & \ar@{..}[ur]           &                        &                             &                     &
}
\]
where the objects $Z^k_i$ lie on the right hand `beam' of the ladder and the notation for the objects $X^k_{i,\infty}$ is explained by Lemma~\ref{lem:homotopy-colimits} below.  The remaining $r$ components, denoted $\mcZ^k_\infty$, are of type $A_1$ as above, where the unique indecomposable complex in each such component is denoted $Z^k_\infty$. The structure of $\strings{\Lambda}$ for $\Lambda$ of infinite global dimension is sketched in Figure~\ref{fig:infinite}.

For $a \in \bZ$, $b \in \bZ \cup \{\infty\}$ and $0 \leq k < r -1$, suspension acts as follows:
\[
\Sigma X^k_{a,b} = X^{k+1}_{a,b},  \quad \Sigma X^{r-1}_{a,b} = X^0_{a+r+m,b+r+m}, \quad 
\Sigma Z^k_b = Z^{k+1}_b , \quad  \Sigma Z^{r-1}_b = Z^0_{b+r+m}.
\]

\begin{remark}\label{rem:correspondence}
We have the following correspondences:
\begin{eqnarray*}
\ind{\mcX}                                        & \bij & \{\text{perfect string complexes}\}; \\
\{X^k_{i,\infty} \mid 0 \leq k \leq r, i \in \bZ\} & \bij & \{\text{left-infinite string complexes}\}; \\
\{Z^k_i \mid 0 \leq k \leq r, i \in \bZ\}         & \bij &  \{\text{right-infinite string complexes}\}; \\
\ind{\mcZ_\infty}                                  & \bij & \{\text{two-sided string complexes}\}.
\end{eqnarray*}
Moreover, in light of Proposition~\ref{compact}, when $\gldim \Lambda = \infty$ the compact indecomposable complexes in $\sK$ are those lying in $\mcX$ components and the `righthand beam' of each of the $\mctZ$ `ladder' components, i.e. the objects $Z^k_i$ for $i \in \bZ$ and $0 \leq k < r$, in $\strings{\Lambda}$.
\end{remark}

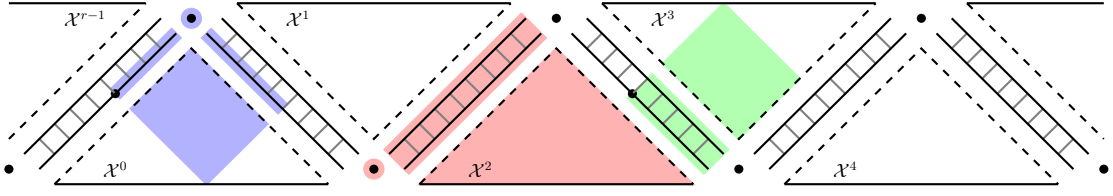
\begin{figure}
\begin{center}
\begin{tikzpicture}[scale=0.2,thick]

\newcommand{\sst}[1]{\scalebox{0.6}{$#1$}}

\draw[fill,blue!30] (4.7,-3.7) -- (5.3,-4.3) -- (9.3,-0.3) -- (8.7, 0.3) -- cycle;
\fill[blue!30] (10,1) circle (0.7);
\draw[fill,blue!30] (6,-5) -- (11,-10) -- (15,-6) -- (10,-1) -- cycle;
\draw[fill,blue!30] (10.7,-0.3) -- (11.3,0.3) -- (16.3,-4.7) -- (15.7,-5.3) -- cycle;
\fill (5,-4) circle (0.3);

\fill[red!30] (22,-9) circle (0.7);
\draw[fill,red!30] (22.7,-7.7) -- (24.3,-9.3) -- (33.3,-0.3) -- (31.7,1.3) -- cycle;
\draw[fill,red!30] (25,-10) -- (34, -1) -- (43,-10) -- cycle;

\draw[fill,green!30] (38.7,-4.3) -- (40.3,-2.7) -- (45.3,-7.7) -- (43.7,-9.3) -- cycle;
\draw[fill,green!30] (41,-2) -- (45,2) -- (50,-3) -- (46,-7) -- cycle;
\fill (39,-4) circle (0.3);

\node at (5,-9)  {\sst{\mcX^0}};  \node at (17,1)  {\sst{\mcX^1}};  \node at (29,-9)  {\sst{\mcX^2}}; \node at (41,1)  {\sst{\mcX^3}}; \node at (53,-9)  {\sst{\mcX^4}}; \node at (3,1)  {\sst{\mcX^{r-1}}};

  \draw[dashed] (-2,-7) -- (7,2);
  \draw (-2,2) -- (7,2);

   \draw[gray] (0,-7) -- (1,-8);
   \draw[gray] (1,-6) -- (2,-7);
   \draw[gray] (2,-5) -- (3,-6);
   \draw[gray] (3,-4) -- (4,-5);
   \draw[gray] (4,-3) -- (5,-4); 
   \draw[gray] (5,-2) -- (6,-3);
   \draw[gray] (6,-1) -- (7,-2);
   \draw[gray] (7,0) -- (8,-1);
   
  \draw (-1,-8) -- (8,1);
   \draw (0,-9) -- (9,0);
   \fill (10,1) circle (0.3);
   \fill (-2,-9) circle (0.3);

  \draw[dashed] (1,-10) -- (10, -1) -- (19,-10);
  \draw         (1,-10) -- (19,-10);

   \draw[gray] (12,-1) -- (13,0);
   \draw[gray] (13,-2) -- (14,-1);
   \draw[gray] (14,-3) -- (15,-2);
   \draw[gray] (15,-4) -- (16,-3);
   \draw[gray] (16,-5) -- (17,-4); 
   \draw[gray] (17,-6) -- (18,-5);
   \draw[gray] (18,-7) -- (19,-6);
   \draw[gray] (19,-8) -- (20,-7);

   \draw (11, 0) -- (20,-9);
   \draw (12,1) -- (21,-8);
   \fill (22,-9) circle (0.3);

  \draw (13,2) -- (31,2);
  \draw[dashed] (13,2) -- (22,-7) -- (31,2);

   \draw[gray] (24,-7) -- (25,-8);
   \draw[gray] (25,-6) -- (26,-7);
   \draw[gray] (26,-5) -- (27,-6);
   \draw[gray] (27,-4) -- (28,-5);
   \draw[gray] (28,-3) -- (29,-4); 
   \draw[gray] (29,-2) -- (30,-3);
   \draw[gray] (30,-1) -- (31,-2);
   \draw[gray] (31,0) -- (32,-1);
   
  \draw (23,-8) -- (32,1);
   \draw (24,-9) -- (33,0);
   \fill (34,1) circle (0.3);

  \draw[dashed] (25,-10) -- (34, -1) -- (43,-10);
  \draw         (25,-10) -- (43,-10);

   \draw[gray] (36,-1) -- (37,0);
   \draw[gray] (37,-2) -- (38,-1);
   \draw[gray] (38,-3) -- (39,-2);
   \draw[gray] (39,-4) -- (40,-3);
   \draw[gray] (40,-5) -- (41,-4); 
   \draw[gray] (41,-6) -- (42,-5);
   \draw[gray] (42,-7) -- (43,-6);
   \draw[gray] (43,-8) -- (44,-7);

   \draw (35, 0) -- (44,-9);
   \draw (36,1) -- (45,-8);
   \fill (46,-9) circle (0.3);

  \draw[dashed] (37, 2) -- (46,-7) -- (55, 2);
  \draw         (37, 2) -- (55, 2);

   \draw[gray] (48,-7) -- (49,-8);
   \draw[gray] (49,-6) -- (50,-7);
   \draw[gray] (50,-5) -- (51,-6);
   \draw[gray] (51,-4) -- (52,-5);
   \draw[gray] (52,-3) -- (53,-4); 
   \draw[gray] (53,-2) -- (54,-3);
   \draw[gray] (54,-1) -- (55,-2);
   \draw[gray] (55,0) -- (56,-1);

  \draw (47,-8) -- (56,1);
   \draw (48,-9) -- (57,0);
   \fill (58,1) circle (0.3);

  \draw[dashed] (49,-10) -- (58, -1) -- (67,-10);
  \draw         (49,-10) -- (67,-10);

   \draw[gray] (60,-1) -- (61,0);
   \draw[gray] (61,-2) -- (62,-1);
   \draw[gray] (62,-3) -- (63,-2);
   \draw[gray] (63,-4) -- (64,-3);
   \draw[gray] (64,-5) -- (65,-4); 
   \draw[gray] (65,-6) -- (66,-5);
   \draw[gray] (66,-7) -- (67,-6);
   \draw[gray] (67,-8) -- (68,-7);

   \draw (59, 0) -- (68,-9);
   \draw (60,1) -- (69,-8);
   \fill (70,-9) circle (0.3);

  \draw[dashed] (61,2) -- (70,-7);
  \draw (61,2) -- (70,2);
\end{tikzpicture}
\end{center}
\caption{The structure of $\strings{\Lambda}$ when $\gldim \Lambda = \infty$. Notation for components is as in Figure~\ref{fig:finite}.}\label{fig:infinite}
\end{figure}

\subsection{Notation for Hom-hammocks} \label{hammocks-notation}

Before describing the structure of the Hom-hammocks in $\strings{\Lambda}$ we first set up some notation.
Let $A \in \ind{\KprojL}$, then we define the Hom-hammocks
\begin{align*}
\homfrom{A} & \coloneqq {} \{ X \in \ind{\KprojL} \mid \Hom_\sK(A,X) \neq 0 \} \\
\homto{A} & \coloneqq {} \{ X \in \ind{\KprojL} \mid \Hom_\sK(X,A) \neq 0 \}.
\end{align*}
For $A \in \strings{\Lambda}$, we define the Hom-hammocks
\begin{align*}
\infhomfrom{A} & \coloneqq {} \{ X \in \strings{\Lambda} \mid \Hom_\sK(A,X) \neq 0 \} \quad \text{the } \textit{forward Hom-hammock} \text{ of } A; \\
\infhomto{A} & \coloneqq {} \{ X \in \strings{\Lambda} \mid \Hom_\sK(X,A) \neq 0 \} \quad \text{the } \textit{backward Hom-hammock} \text{ of } A.
\end{align*}

\subsection{Hom-hammocks in $\strings{\Lambda}$ for $\Lambda$ of finite global dimension}

 We introduce the following notation to encode the action of the suspension on indecomposable objects.

\begin{notation} \label{notation}
For $a \in \bZ$  we set
\[
a' = \left\{
\begin{array}{ll}
a +r + m & \text{if } k = r-1; \\
a        & \text{ otherwise,}
\end{array}
\right. 
\quad 
a'' = \left\{
\begin{array}{ll}
a +r - n & \text{if } k = r-1; \\
a       & \text{ otherwise,}
\end{array}
\right. 
\]
\end{notation}

The following proposition will be more easily understood by looking at Figure~\ref{fig:finite} and natural extensions of the figures in \cite[\S 2]{BPP1}.

\begin{proposition} \label{prop:fin-hammocks}
Suppose $\Lambda$ is derived-discrete of finite global dimension. Let $a,b \in \bZ$ and $0 \leq k < r$. The forward Hom-hammocks of objects of $\strings{\Lambda}$ are given by:
\begin{align*}
\infhomfrom{X^k_{a,b}} = {} & \homfrom{X^k_{a,b}} \cup \{ X^k_{i,\infty} \mid a \leq i \leq b \} \cup  \{ X^{k+1}_{-\infty,j} \mid a' - 1 \leq j \leq b' - 1\}; \\
\infhomfrom{Y^k_{a,b}} = {} & \homfrom{Y^k_{a,b}} \cup \{ Y^k_{\infty,j} \mid b \leq j \leq a \} \cup \{ Y^{k+1}_{i,-\infty} \mid b''-1 \leq i \leq a'' - 1\}; \\
\begin{split}
\infhomfrom{Z^k_{a,b}} = {} & \homfrom{Z^k_{a,b}} \cup \{ X^{k+1}_{-\infty,j} \mid j \geq a' - 1\} \cup \{X^{k+1}_{i,\infty} \mid i \leq a'-1 \}\\
                                        & \cup \{ Z^{k+1}_\infty \} \cup \{ Y^{k+1}_{i,-\infty} \mid i \geq b''-1 \} \cup \{ Y^{k+1}_{\infty,j} \mid j \leq b''-1 \};
\end{split} \\
\begin{split}
\infhomfrom{X^k_{a,\infty}} = {} & \{ X^k_{i,\infty} \mid i \geq a \} \cup \{ Z^k_{i,j} \mid  i \geq a \text{ and } j \in \bZ\} \cup \{ X^{k+1}_{-\infty,j} \mid j \geq a' - 1\} \\
                                        & \cup \{ X^{k+1}_{i,j} \mid i \leq a'-1 \text{ and } j \geq a'-1 \};
\end{split} \\
\begin{split}
\infhomfrom{X^k_{-\infty,b}} = {} & \{ X^k_{-\infty,j} \mid j \geq b \} \cup \{ X^k_{i,j} \mid i \leq b \text{ and } j \geq b \}  \cup \{ X^k_{i,\infty} \mid i \leq b \} \\
                                        & \cup \{ Y^k_{\infty,j} \mid j \in \bZ \} \cup \{ Z^k_{i,j} \mid  i \leq b \text{ and } j \in \bZ\} \cup \{ Z^k_\infty \};
\end{split} \\
\begin{split}
\infhomfrom{Y^k_{\infty,b}} = {} & \{ Y^k_{\infty,j} \mid j \geq b \} \cup \{ Z^k_{i,j} \mid  i \in \bZ \text{ and } j \geq b \} \cup \{ Y^{k+1}_{i,-\infty} \mid i \geq b'' - 1\} \\
                                        & \cup \{ Y^{k+1}_{i,j} \mid i \geq b''-1 \text{ and } j \leq b''-1 \};
\end{split} \\
\begin{split}
\infhomfrom{Y^k_{a,-\infty}} = {} & \{ Y^k_{i,-\infty} \mid i \geq a \} \cup \{ Y^k_{i,j} \mid i \geq a \text{ and } j \leq a \}  \cup \{ Y^k_{\infty,j} \mid j \leq a \} \\
                                        & \cup \{ X^k_{i,\infty} \mid i \in \bZ \} \cup \{ Z^k_{i,j} \mid  i \in \bZ \text{ and } j \leq a \} \cup \{ Z^k_\infty \};
\end{split} \\
\infhomfrom{Z^k_\infty} = {} & \{X^k_{i,\infty} \mid i \in \bZ \} \cup \{ Y^k_{\infty,j} \mid j \in \bZ \} \cup \{ Z^k_{i,j} \mid i, j \in \bZ \}.
\end{align*}
\end{proposition}

The Hom-hammocks involving only perfect string complexes were established in \cite[\S2 \& \S 5]{BPP1}; we extend these to left- and right-infinite and two-sided string complexes here. We start by extending the `extended rays' and `extended corays' of \cite[Properties 1.2(5)]{BPP1}.

\begin{lemma} \label{rays-and-corays}
There are sequences of objects and morphisms in $\strings{\Lambda}$:
\[ 
\xymatrix@C=0.7em@R=1ex{
   X^k_{ii}    \ar[r] & X^k_{i,i+1} \ar[r] & \ccdots \ar[r] &
   X^k_{i,\infty} \ar[r] & \ccdots \ar[r] & Z^k_{i,i} \ar[r] & \ccdots \ar[r] &
   \Sigma X^k_{-\infty,i-1} \ar[r] & \ccdots \ar[r] & \Sigma X^k_{i-2,i-1} \ar[r] & \Sigma X^k_{i-1,i-1},
\\
   Y^k_{ii} \ar[r] & Y^k_{i+1,i} \ar[r] & \ccdots    \ar[r] &
   Y^k_{\infty,i} \ar[r] & \ccdots \ar[r] & Z^k_{i,i} \ar[r] & \ccdots \ar[r] &
   \Sigma Y^k_{i-1,-\infty} \ar[r] & \ccdots \ar[r] & \Sigma Y^k_{i-1,i-2} \ar[r] & \Sigma Y^k_{i-1,i-1},
}
\]
\vspace{-1mm} 
\[
\xymatrix@C=0.7em@R=1ex{
   \ccdots \ar[r] & X^k_{-\infty,i-1} \ar[r] & X^k_{-\infty,i} \ar[r] &
   X^k_{-\infty,i+1} \ar[r] & \ccdots \ar[r] & Z^k_\infty \ar[r] & \ccdots \ar[r] &
   Y^k_{i-1,\infty} \ar[r] & Y^k_{i,\infty} \ar[r] & Y^k_{i+1,\infty} \ar[r] & \ccdots,
\\
\ccdots \ar[r] & Y^k_{i-1,-\infty} \ar[r] & Y^k_{i,-\infty} \ar[r] &
   Y^k_{i+1,-\infty} \ar[r] & \ccdots \ar[r] & Z^k_\infty \ar[r] & \ccdots \ar[r] &
   X^k_{\infty,i-1} \ar[r] & X^k_{\infty,i} \ar[r] & X^k_{\infty,i+1} \ar[r] & \ccdots,   
} 
\]
such that any composition is nonzero.
\end{lemma}

The sequences in Lemma~\ref{rays-and-corays} whose leftmost parts consist of objects of $\mcX \cup \mcX_{-\infty}$ will be called \emph{extended rays} in $\strings{\Lambda}$ and those whose leftmost parts consist of objects of $\mcY \cup \mcY_{-\infty}$ will be called \emph{extended corays}. 

\begin{proof}
This is an exercise in homotopy string combinatorics using Bobi\'nski's algorithm from \cite{Bo} or \cite[\S 6]{ALP} and the description of morphisms between string complexes given in \cite[Thm.~3.15]{ALP}, cf. \cite[\S 7]{ALP} for a more concrete description in this setting, and \cite[Appendix B]{BPP1} using classical string combinatorics.
\end{proof}

We now identify some distinguished triangles in $\sK$ involving objects of $\strings{\Lambda}$ analogous to those in \cite[Properties 1.2(4)]{BPP1}.

\begin{lemma} \label{triangles}
Let $a, b \in \bZ$. The following are distinguished triangles in $\sK$:
\begin{align*}
\tri{X^k_{a,a}}{X^k_{a,\infty}}{X^k_{a+1,\infty}} \quad & \text{and} \quad \tri{X^k_{a,a}}{X^{k+1}_{-\infty,a-1}}{X^{k+1}_{-\infty,a}} \\
\tri{Y^k_{a,a}}{Y^k_{\infty,a}}{Y^k_{\infty,a+1}} \quad & \text{and} \quad \tri{Y^k_{a,a}}{Y^{k+1}_{a-1,-\infty}}{Y^{k+1}_{a,-\infty}} \\
\tri{X^k_{-\infty,b}}{Z^k_\infty}{X^k_{b+1,\infty}} \quad & \text{and} \quad \tri{Y^k_{a,-\infty}}{Z^k_{\infty}}{Y^k_{\infty,a+1}}.
\end{align*}
\end{lemma}

\begin{proof}
From the form of the homotopy strings in Lemma~\ref{lem:strings}, and therefore the straightforward form of the corresponding string complexes, the mapping cones of the morphisms in Lemma~\ref{rays-and-corays} is a standard computation.
\end{proof}

We are now ready to prove Proposition~\ref{prop:fin-hammocks}. The arguments are in the spirit of \cite[\S 2]{BPP1} so we just give a sketch.

\begin{proof}[Proof of Proposition~\ref{prop:fin-hammocks}]
Let $A$ and $B$ be string complexes. We split the argument up into different cases.

\Case{$A \in \ind{\mcX \cup \mcY}$ and $B \in \strings{\Lambda}$} 
In fact, we consider only the case $A \in \ind{\mcX}$; the case $A \in \ind{\mcY}$ is analogous. Suppose $A = X^k_{a,b}$. We first deal with the non-vanishing statements. If $B$ is a perfect string complex then this is \cite[Prop.~2.4]{BPP1}, so we assume that $B \in \ind{\mcX^k_\infty \cup \mcX^{k+1}_{-\infty}}$. An induction on the height of $A = X^k_{a,b}$, $h(A) = b - a$, then gives the desired conclusion. Use Lemma~\ref{rays-and-corays} for the base step. For the inductive step, apply $\Hom_{\sK}(-,X^k_{i,\infty})$ and $\Hom_{\sK}(-,X^{k+1}_{-\infty,i})$ to the triangles from \cite[Lem.~2.2]{BPP1}
\[
{}_0 A \to A \to A'' \to \Sigma ({}_0 A) 
\quad \text{and} \quad
A' \to A \to A_0 \to \Sigma A',
\]
where ${}_0 A = X^k_{a,a}$, $A' = X^k_{a,b-1}$, $A_0 = X^k_{b,b}$ and $A'' = X^k_{a+1,b}$, and read off the Hom-spaces from the resulting long exact sequences. The only subtlety is when $r = 1$ and $\Sigma B = \Sigma X_{i,\infty} = X_{b,\infty}$ (the superscripts are dropped because $r=1$). In this case $i = b-1-m$ and we cannot infer that $\Hom_{\sK}(\Sigma^{-1} A_0,B) = \Hom_{\sK}(\Sigma^{-1} X_{b,b},X_{b-1-m,\infty}) =0$. However, by Lemma~\ref{rays-and-corays} the composition $\Sigma^{-1} X_{b,b} \to A' = X_{a,b-1} \to X_{b-1-m,\infty}$ factors as $\Sigma^{-1} X_{b,b} \to X_{b-1-m,\infty} \to X_{a,b-1} \to X_{b-1-m,\infty}$. Since $\dim \Hom_{\sK}(X_{b-1-m,\infty},X_{b-1-m,\infty}) = 1$, we have that the composition must be zero, for otherwise $X_{b-1-m,\infty}$ would be a summand of $X_{a,b-1}$; a contradiction. This now gives the required non-vanishing statement.
The vanishing statements are obtained similarly by induction.

\Case{$A \in \ind{\mcZ}$ and $B \in \strings{\Lambda}$}
Let $A = Z^k_{a,b}$. We start with $B \in \ind{\mcX^{k+1}_{-\infty}}$. Recalling Notation~\ref{notation}, Lemma~\ref{triangles} says that for each $n \in \bN$ there is a distinguished triangle, 
\begin{equation} \label{triangle-Xinf}
\trilabels{X^{k+1}_{-\infty,a'-1 +n}}{X^{k+1}_{-\infty,a+n}}{X^{k+1}_{a'-1+n,a'-1+n}}{}{}{-\Sigma f}.
\end{equation}
By Lemma~\ref{rays-and-corays} $\Hom_{\sK}(Z^k_{a,b},X^{k+1}_{-\infty,a'-1}) \neq 0$. Applying $\Hom_{\sK}(Z^k_{a,b},-)$ to \eqref{triangle-Xinf} allows one to read off $\Hom_{\sK}(Z^k_{a,b},X^{k+1}_{-\infty,j}) \neq 0$ inductively for all $j \geq a' - 1$. The only subtlety is when $r = 1$ and $m= 0$, i.e. when $\mcX$ is a `$0$-spherical' component. In this case we get
\[
(Z_{a,b},\Sigma^{-1} X_{a+1,a+1}) \rightlabel{(Z_{a,b},f)} (Z_{a,b}, X_{-\infty,a}) \too (Z_{a,b}, X_{-\infty,a+1}) \too (Z_{a,b},X_{a+1,a+1}),
\]
where we have dropped the superscript $k$ because $r = 1$.
By Lemma~\ref{rays-and-corays}, $f$ decomposes as $\Sigma^{-1} X_{a+1,a+1} = X_{a,a} \rightlabel{f_1} Z_{a,b} \rightlabel{f_2} X_{-\infty,a}$. Now, $(Z_{a,b},f)(g) = fg = f_2 f_1 g$. By \cite[\S 2 \& \S5]{BPP1}, we have $f_1 g \colon Z_{a,b} \to Z_{a,b}$ is an isomorphism or zero. If it were an isomorphism then $Z_{a,b}$ would be a summand of $X_{a,a}$; a contradiction. Therefore $fg = 0$ and $(Z_{a,b},f) =0$.

For the vanishing statements in the $\mcX^{k+1}_{-\infty}$ components, use \eqref{triangle-Xinf} with $n \in \bZ_{\leq 0}$, using the one-dimensionality of the nonzero Hom-spaces to start the induction. The same argument applied to the $\mcX^l_{-\infty}$ components for $l \neq k+1 \modulo r$ shows $\Hom_{\sK}(Z^k_{a,b},-)=0$ on those components. Similar arguments can be used for the $\mcY^l_{-\infty}$ components.

Next we consider $B \in \ind{\mcX^{k+1}_\infty}$. By Lemma~\ref{triangles}, for each $i \in \bZ$ we have
\[
\Sigma^{-1} Z^{k+1}_{i,b''} \rightlabel{f} Y^{k+1}_{b''-1,-\infty} \too X^{k+1}_{i,\infty} \too Z^{k+1}_{i,b''}.
\]
Applying the functor $\Hom_{\sK}(Z^k_{a,b},-)$ to this triangle gives the long exact sequence,
\[
(Z^k_{a,b}, \Sigma^{-1} Z^{k+1}_{i,b''}) \rightlabel{(Z^k_{a,b},f)} (Z^k_{a,b}, Y^{k+1}_{b''-1,-\infty}) \too (Z^k_{a,b}, X^{k+1}_{i,\infty}) \too (Z^k_{a,b}, Z^{k+1}_{i,b''}).
\]
Now by \cite[\S2]{BPP1}, $\Hom_{\sK}(Z^k_{a,b}, Z^{k+1}_{i,b''}) = 0$ for all $i \in \bZ$. For $i < a'$, \cite[\S2]{BPP1} implies that $\Hom_{\sK}(Z^k_{a,b}, \Sigma^{-1} Z^{k+1}_{i,b''}) = 0$ giving the non-vanishing statement. For $i \geq a'$, one-dimensionality and Lemma~\ref{rays-and-corays} give that the map $(Z^k_{a,b},f)$ is a surjection, whence we get the vanishing statements for the component $\mcX^{k+1}_\infty$. The vanishing statements for the $\mcX^l_\infty$ components for $l \neq k+1$ follow from the vanishing statements for the corresponding $\mcY^l_{-\infty}$ components in an analogous manner. Similarly for the $\mcY^l_\infty$ components.

Finally, the statements for $Z^l_\infty$ can be deduced by applying the functor $\Hom_{\sK}(Z^k_{a,b},-)$ to the triangles $\tri{X^k_{-\infty,a'-1}}{Z^k_\infty}{X^k_{a',\infty}}$ from  Lemma~\ref{triangles}.

\Case{$A$ is a nonperfect string complex and $B \in \strings{\Lambda}$}
If $B$ is perfect, then this case is the dual of the two cases above, so we may assume that $B$ is also nonperfect.
The non-vanishing statements are contained in Lemma~\ref{rays-and-corays}.
For the vanishing statements, one needs to argue once from string combinatorics and \cite[Thm.~3.15]{ALP} for each $\mcX_{\pm \infty}$ and $\mcY_{\pm \infty}$ component and then use the triangles from Lemma~\ref{triangles} for an induction.
\end{proof}

\subsection{Hom-hammocks in $\strings{\Lambda}$ for $\Lambda$ of infinite global dimension}

We now state the analogue of Proposition~\ref{prop:fin-hammocks} for $\Lambda$ of infinite global dimension. The proof is analogous to that for Proposition~\ref{prop:fin-hammocks} and is therefore omitted. The statement is more easily understood with reference to a figure, the relevant illustrations are Figures~\ref{fig:left-infinite} and \ref{fig:infinite}. Recall Notation~\ref{notation}.

\begin{proposition} \label{prop:inf-hammocks}
Suppose $\Lambda$ is derived-discrete of infinite global dimension. Let $a \leq b \in \bZ$ and $0 \leq k < r$. The forward Hom-hammocks of objects of $\strings{\Lambda}$ are given by:
\begin{align*}
\infhomfrom{X^k_{a,b}} = {} & \homfrom{X^k_{a,b}} \cup \{ X^k_{i,\infty} \mid a \leq i \leq b \}; \\
\infhomfrom{Z^k_a} = {} & \homfrom{Z^k_a} \cup \{ Z^{k+1}_\infty \} \cup \{ X^{k+1}_{i,\infty} \mid i \leq a'-1 \}; \\
\infhomfrom{X^k_{a,\infty}} = {} & \homfrom{Z^k_a} \cup \{ Z^{k+1}_\infty \} \cup \{ X^k_{i,\infty} \mid i \geq a \}. 
\end{align*}
\end{proposition}

In this case there is a little asymmetry in the definitions of the Hom-hammock regions $\infhomto{X^k_{a,\infty}}$ and $\infhomto{Z^k_a}$, so we make an explicit note of them here for $0 \leq k < r$: 
\[
\infhomto{X^k_{a,\infty}} = \infhomfrom{Z^{k-1}_{\bar{a}}}
\text{ and }
\infhomto{Z^k_a} = \homto{Z^k_a} \cup \{ Z^k_\infty \} \cup \{ X^k_{i,\infty} \mid i \leq a \},
\]
where
$\bar{a} = \left\{
\begin{array}{ll}
a - r - m & \text{if } k = 0; \\
a         & \text{otherwise.}
\end{array}
\right.$

\subsection{Factorisation properties}

Finally, in order to determine how the simple functors isolate indecomposable pure-injective objects in the next section, we need to understand how morphisms factor in $\strings{\Lambda}$.

\begin{proposition} \label{prop:factorisations}
Let $A, B, C \in \strings{\Lambda}$. If $B \in \infhomfrom{A}$ and $C \in \infhomfrom{A} \cap \infhomfrom{B}$ then any map $f \colon A \to C$ factors as $A \to B \to C$.
\end{proposition}

\begin{proof}
If $A,B, C \in \ind{\KprojL}$ then this statement can be deduced from \cite[\S 2]{BPP1} or \cite[\S 4]{BK}. If $A$, $B$ and $C$ lie on the same extended ray or coray, then this is Lemma~\ref{rays-and-corays}. The rest of the proof splits up into a case analysis building up from these rays and corays.

\Case{1} \emph{$A \in \ind{\mcX \cup \mcY}$.} We assume $A \in \ind{\mcX^k}$ for some $0 \leq k < r$; the case $A \in \ind{\mcY^k}$ is analogous. Suppose $C \in \ind{\mcZ^k \cup \mcX^{k+1}}$. We verify the factorisation for $B \in \mcX^k_\infty$; the check for $B \in \mcX^{k+1}_{-\infty}$ is analogous. There exist $B' \in \ind{\mcX^k} \cap \infhomfrom{A} \cap \infhomto{C}$ and $C' \in \ind{\mcZ^k \cup \mcX^{k+1}} \cap \infhomfrom{A} \cap \infhomto{C}$ such that $B'$, $B$ and $C'$ lie on the same extended ray or coray. The map $A \to C$ thus factors as $A \to B' \to C' \to C$ by the factorisation statements for maps between indecomposable perfect complexes. However, by Lemma~\ref{rays-and-corays}, the map $B' \to C'$ factors as $B' \to B \to C'$ giving the desired statement. If $B, C \in \ind{\mcX^k_\infty}$ (or  $B, C \in \ind{\mcX^{k+1}_{-\infty}}$) then this is the dual of Case 3 below.

\Case{2} \emph{$A \in \ind{\mcZ}$.} If $C \in \ind{\mcX \cup \mcY}$ then this is the dual to Case 1. Assume $A \in \ind{\mcZ^k}$. We only need to check the case that $C \in \ind{\mcZ^{k+1}}$ (the other cases are covered by Cases 3 and 4 and their duals). Suppose $B \in \ind{\mcX^{k+1}_{-\infty}}$. Then there exists $B' \in \ind{\mcX^{k+1}} \cap \infhomfrom{A} \cap \infhomfrom{B} \cap \infhomto{C}$. In particular, the map $A \to C$ factorises as $A \to B' \to C$ and by above the map $A \to B'$ factorises as $A \to B \to B'$, as required. 

\Case{3} \emph{$A \in \ind{\mcX_{\pm \infty} \cup \mcY_{\pm \infty}}$.} We treat the case $A \in \mcX^k_\infty$; the other cases are analogous. For $B$ and $C$ lying on the same extended ray or coray, this is again Lemma~\ref{rays-and-corays}, so assume $B$ and $C$ lie on a different extended ray or coray from $A$. We first cover the case that $B \in \ind{\mcX^k_\infty}$. Assume $A = X^k_{a,\infty}$ for some $a \in \bZ$ and $B = X^k_{b,\infty}$ for some $b > a$. Inductively using triangles of the form
\[
\tri{X^k_{i,i}}{X^k_{i,\infty}}{X^k_{i+1,\infty}}
\]
with $a \leq i < b$ and using the fact that if $C \in \infhomfrom{B}$ then by Proposition~\ref{prop:fin-hammocks} and its dual, we have $\Hom_{\sK}(X^k_{i,i},C) = 0$, giving the desired factorisation. A similar argument holds if $B \in \ind{\mcX^{k+1}_{-\infty}}$ using the analogous triangles. For $B$ elsewhere, apply the same argument, but one should use the triangles from \cite[Properties 1.1(4)]{BPP1} or \cite[Lem.~2.2]{BPP1}.

\Case{4} \emph{$A = Z^k_\infty$.} Apply the same argument as in Case 3 using the triangles $\tri{X^k_{-\infty,b}}{Z^k_\infty}{X^k_{b+1,\infty}}$ or  $\tri{Y^k_{a,-\infty}}{Z^k_\infty}{X^k_{\infty,a+1}}$ for appropriate choices of $a$ and $b$ to get factorisation through the objects of $\ind{\mcX^k_\infty \cup \mcY^k_\infty}$. For $B \in \ind{\mcZ^k}$ one can get a factorisation $A \to B' \to C$ for an appropriate choice of $B' \in \ind{\mcX^k_\infty \cup \mcY^k_\infty}$, which further factorises as $B' \to B \to C$, giving the desired factorisation.
\end{proof}

\subsection{Homotopy colimits}

It is interesting to identify the indecomposable objects lying on the $\mcX^k_{\pm\infty}$ and $\mcY^k_{\pm\infty}$ as homotopy colimits; see \cite{Neeman-book} for the definition. The proofs of the following statements are an exercise in homotopy string combinatorics; since the statements are not needed elsewhere in the paper, we leave the proofs to the reader.

\begin{lemma} \label{lem:homotopy-colimits}
Suppose $\Lambda$ is derived-discrete of finite global dimension. Then:
\begin{enumerate}
\item For each chain of objects and morphisms along a ray in an $\mcX$ component,  $X^k_{a,a} \to X^k_{a,a+1} \to X^k_{a,a+2} \to \cdots$, we have $X^k_{a,\infty} \cong \hocolim X^k_{a,j}$.
\item For each chain of objects and morphisms along a coray in a $\mcY$ component,  $Y^k_{b,b} \to Y^k_{b+1,b} \to Y^k_{b+2,b} \to \cdots$, we have $Y^k_{\infty,b} \cong \hocolim Y^k_{i,b}$.
\item For each chain of objects and morphisms along a ray in a $\mcZ$ component,  $Z^k_{a,j} \to Z^k_{a,j+1} \to Z^k_{a,j+2} \to \cdots$, we have $X^{k+1}_{-\infty,a-1} \cong \hocolim Z^k_{a,j}$.
\item For each chain of objects and morphisms along a coray in a $\mcZ$ component,  $Z^k_{i,b} \to Z^k_{i+1,b} \to Z^k_{i+2,b} \to \cdots$, we have $Y^{k+1}_{b-1,-\infty} \cong \hocolim Z^k_{i,b}$.
\end{enumerate}
Note that statement (1) also holds in the case that $\Lambda$ is derived-discrete of infinite global dimension.
\end{lemma}

\section{Krull-Gabriel dimension and simple functors} \label{KGdim-simples}

The Krull-Gabriel dimension of $\sD(\Mod{\Lambda})$, that is, of $\Coh{\sD(\Mod{\Lambda})}$, was computed in \cite{BK}. In this section we follow the approach of \cite{BK} to compute the Krull-Gabriel dimension of $\Coh{\sK}$ and to identify the simple functors in $\Coh{\sK}/\Coh{\sK}_n$ for each $n \in \bN$.   
The main result of this section is the following, where the first statement is Theorem~\ref{intro:zgrank} and the second statement is part of Theorem~\ref{intro:indecomposable}.
Note the contrast with $\KG(\sD(\Mod{\Lambda})) = 1$ in the infinite global dimension case in \cite{BK}.

\begin{theorem} \label{thm:KG-dim}
Let $\Lambda$ be a derived-discrete algebra. Then 
\begin{enumerate}
\item $\KG(\KProjL) = \CB(\sK) = 2$.
\item The objects of $\strings{\Lambda}$ form a complete list of indecomposable, pure-injective complexes in $\sK$.
\end{enumerate}
\end{theorem}

Using the same argument as \cite[Lem. 10.2.2]{PreNBK}, we have the following useful characterisation of coherent subfunctors of $F_f$.

\begin{lemma} \label{lem:coherent-subfunctors}
Let $\T$ be a compactly generated triangulated category and let $F_f$ be a coherent functor.  Then, all coherent subfunctors of $F_f$ have the form $\image(h,-)/\image(f,-)$ for some factorisation $f = gh$ of $f$ in $\Tc$.
\end{lemma}

\begin{lemma} \label{lem:KG-defined}
Let $\Lambda$ be a derived-discrete algebra. Then $\KG(\sK)$ is defined.
\end{lemma}

\begin{proof}
Let $F \in \Coh{\sK}$ and consider the lattice $L(F)$ of coherent subfunctors of $F$.  By Lemma~\ref{lem:coherent-subfunctors} and the descriptions of the Hom-hammocks in Propositions~\ref{prop:fin-hammocks} and \ref{prop:inf-hammocks}, it is clear that $L(F)$ has no densely ordered subset. By Remark~\ref{KGdimmdim}, it follows that $\KG(\sK)$ is defined.
\end{proof}

By Lemma \ref{lem:isolation} and Proposition \ref{prop:KG-CB}, we have 
\[ 
\Coh{\sK}/\Coh{\sK}_n = \Coh{\sK}_{\cl{X}_n}
\] 
for all $n \geq 0$, where $\cl{X}_n$ is the closed subset of $\Zg(\sK)$ consisting of points with CB-rank greater than or equal to $n$.  Let $q_n \colon\Coh{\sK} \to \Coh{\sK}_{\cl{X}_n}$ denote the corresponding localisation functor. The following is immediate from Lemma~\ref{lem:isolation} and Proposition~\ref{prop:KG-CB}.

\begin{corollary}\label{cor:simples-isol}
The CB-rank of $\Zg(\sK)$ is defined and there is a natural bijection
\[ \{ M \in \Zg(\sK) \mid \CB(M) = n \} \bij \{ F \in \Coh{\sK}_{\cl{X}_n} \mid F \text{ is simple}\}. \]
\end{corollary}

\subsection{Cantor-Bendixson rank $0$}  It is well-known that the simple functors in $\Coh{\sK}$ correspond to the Auslander-Reiten triangles comprised of compact objects in ${\sK}$; see \cite[\S 2]{Auslander}. 

\begin{proposition}
The simple objects in $\Coh{\sK}$ are exactly those of the form $F_f$ where $X \overset{f}{\rightarrow} Y \overset{g}{\rightarrow} Z \rightarrow \Sigma X$ is an Auslander-Reiten triangle.  The point of $\Zg(\sK)$ isolated by $F_f\cong (X,-)/{\rm im}(f,-)$ is $X$.
\end{proposition}

\begin{corollary}\label{cor:isol} 
Let $\Lambda$ be a derived-discrete algebra
\begin{enumerate}
\item If $\gldim \Lambda = \infty$ then $\ind{\mcX}$ is the set of isolated points in $\Zg(\sK)$.
\item If $\gldim \Lambda < \infty$ then $\ind{\mcX \cup \mcY \cup \mcZ}$ is the set of isolated points in $\Zg(\sK)$.
\end{enumerate}
\end{corollary}

\begin{proof}
This now follows from Corollary~\ref{cor:simples-isol} and the description of Auslander-Reiten triangles in \cite{BGS} from which it follows that it is the indecomposable perfect complexes which begin Auslander-Reiten sequences.
\end{proof}

\subsection{Cantor-Bendixson rank $1$}  We begin by identifying certain morphisms that give rise to simple functors in $\Coh{\sK}_{\cl{X}_0}$.  

\begin{definition} \label{def:1-simple}
Let $\Lambda$ be a derived-discrete algebra and recall Notation~\ref{notation}.
\begin{enumerate}
\item Suppose $\gldim \Lambda < \infty$.  For $0 \leq k < r$ the morphisms 
\begin{align*}
& h \colon X^k_{i,j} \to X^k_{i+1, j} \oplus Z^k_{i,t}   & \quad &  h \colon Y^k_{i,j} \to Y^k_{i,j+1} \oplus Z^k_{t, j} \\ 
& h \colon Z^k_{i,j} \to Z^k_{i+1, j} \oplus X^{k+1}_{t,i'-1} & \quad & h \colon Z^k_{i,j} \to Z^k_{i, j+1} \oplus Y^{k+1}_{j''-1,t}
\end{align*} 
where $h = \left( \begin{smallmatrix} h_1 \\ h_2 \end{smallmatrix} \right)$ with $h_1 \neq 0$ and $h_2 \neq 0$ will be called \emph{$1$-simple morphisms}.
\item Suppose $\gldim \Lambda = \infty$.  For $0 \leq k < r$ the morphisms
\begin{align*}
&h \colon X^k_{i,j} \to X^k_{i+1, j} \oplus Z^k_i & \quad & h \colon Z^k_j \to Z^k_{j+1} \oplus X^{k+1}_{i, j'-1}
\end{align*}
where $h = \left( \begin{smallmatrix} h_1 \\ h_2 \end{smallmatrix} \right)$ with $h_1 \neq 0$ and $h_2 \neq 0$ will be called \emph{$1$-simple morphisms}.
\end{enumerate}
\end{definition}  

From Lemma~\ref{lem:coherent-subfunctors}, which correlates subfunctors of $F_h$ with factorisations of $h$, and our description of morphisms, it is clear that $q_0(F_h)$ is a simple functor if $h$ is a $1$-simple morphism.   Since these functors are simple, it follows from Corollary~\ref{cor:simples-isol} that there are corresponding indecomposable pure-injective objects with CB-rank equal to $1$.

\begin{proposition}\label{prop:CB1pureinj}  
Let $\Lambda$ be a derived-discrete algebra.
\begin{enumerate}
\item If $\gldim \Lambda = \infty$ then the set $\ind{\mctZ}$  is a complete list of indecomposable, pure-injective complexes of CB rank $1$.
\item If $\gldim \Lambda < \infty$ then the set $\ind{\mcX_\infty \cup \mcX_{-\infty} \cup \mcY_\infty \cup \mcY_{-\infty}}$ is a complete list of indecomposable, pure-injective complexes of CB-rank $1$.
\end{enumerate}
\end{proposition}

\begin{lemma}
In the set-up of Proposition \ref{prop:CB1pureinj}, the complexes listed in each case are indecomposable, pure-injective and of CB-rank $1$.
\end{lemma}
\begin{proof}
For each of these complexes there is a $1$-simple morphism $f$ such that the open set corresponding to the functor $(X,-)/\image{(f,-)}$, where $X$ is the domain of $f$, contains just that complex and complexes of CB-rank $0$.  So this is immediate from Proposition~\ref{prop:KG-CB} together with the description of the Hom-hammocks in Propositions~\ref{prop:fin-hammocks} and \ref{prop:inf-hammocks}.
\end{proof}

To prove that the list is complete, we first need a definition and some preliminary results.

\begin{definition}
If $h \colon A \to B_1 \oplus B_2$ and $h'\colon C \to D_1\oplus D_2$ are $1$-simple morphisms, then we say that $h \sim h'$ if and only if $A$ and $C$ are contained the same ray and $B_2$ and $D_2$ are contained in the same coray of the AR quiver.  Note: when $\Lambda$ has infinite global dimension, objects on the righthand beam of the $\mctZ$ components are on the same ray or coray if and only if they are equal.
This is clearly an equivalence relation on the set of $1$-simple morphisms.
\end{definition}

From the description of the Hom-hammocks in Propositions~\ref{prop:fin-hammocks} and \ref{prop:inf-hammocks}, we see that $(F_h)$ and $(F_g)$ contain the same pure-injective with CB-rank $1$ exactly when $h \sim g$.  So the following corollary is immediate from Corollary~\ref{cor:simples-isol} and inspection of the Hom-hammocks.

\begin{corollary}
Let $h$ and $g$ be $1$-simple morphisms.  Then, $q_0(F_{h}) = q_0(F_g)$ if and only if $h \sim g$.
\end{corollary}

Next we will prove that all simple functors in $  \Coh{\sK}_{\cl{X}_0}$ arise from $1$-simple morphisms. 

\begin{lemma}\label{lem:indecdom}
Let $\T$ be a compactly generated triangulated category and let $f \colon A \to B$ be a morphism in $\Tc$.  If $q \colon \Coh{\T} \to \Coh{\T}_{\cl{X}}$ is a localisation functor and $q(F_f)$ is simple, then there exists some $g \colon C \to D$ in $\Tc$ such that $q(F_f) = q(F_g)$ and $C$ is indecomposable.
\end{lemma}

\begin{proof}
Suppose $A = A' \oplus A''$ where $A', A'' \in \Tc$ are nonzero.  Since $(A, -) \cong (A', -) \oplus (A'',-)$, we have $F_f \cong \big((A',-) \oplus (A'',-)\big) / \image(f,-)$ where we identify $\image(f,-)$ with its image in $(A',-)\oplus (A'',-)$.  Consider the subfunctors of $F_f$ 
\[ {^{(\image(f,-)+A)}/_{\image(f,-)}} \cong {^{(A',-)}/_{\image(f,-) \cap (A',-)}} =H' \] and \[
{^{(\image(f,-)+A)}/_{\image(f,-)}} \cong {^{(A'',-)}/_{\image(f,-) \cap (A'',-)}  =H''}. \]  
The sum of these subfunctors is $F_f$ so, since $q(F_f)$ is simple, the image of at least one of them under $q$ equals $q(F_f)$.  Hence either $q(F_f) \cong q(H)$ or $q(F_f) \cong q(H')$.

Now, $\image(f,-) \cap (A',-) \subseteq (A',-)$ is a finitely generated subfunctor so, by Lemma~\ref{lem:coherent-subfunctors}, there exists some $g \colon A' \to B'$ in $\Tc$ such that $\image(g,-) \cong \image(f,-) \cap (A',-)$ and similarly with $A''$ in place of $A$.  As $\Tc$ is Krull-Schmidt, the result follows.
\end{proof}

\begin{proposition}\label{prop: 1-simples}
Suppose $q_0(F_f)$ is simple in $\Coh{\sK}_{\cl{X}_0}$, then there exists a $1$-simple morphism $h$ such that $q_0(F_f) = q_0(F_h)$. 
\end{proposition}

The proof for the finite global dimension case can be found in \cite{BK}; our proof applies to both infinite and finite global dimension. 

\begin{proof}
If $f = gh$, then $F_h $ is a factor of $ F_f$.  Thus, for any such $h$, we have $q_0(F_h) \cong q_0(F_f)$ if and only if $q_0(F_h) \neq 0$.  We have already observed that if $h$ is a $1$-simple morphism, then $q_0(F_h) \neq 0$ so it remains to show that we always have a factorisation $f = gh$ where $h$ is $1$-simple.  By Lemma~\ref{lem:indecdom} we may assume that $f \colon A \to \bigoplus_{i=1}^n B_i$ where $A, B_1, \dots , B_n$ are indecomposable objects in $\Tc$.  

We observe that the Hom-hammock structure, combined with Proposition~\ref{prop:factorisations}, implies that there is a $1$-simple morphism through which $f$ factors, except in the following cases:
\begin{enumerate}
\item $A = X^k_{i,j}$ for some $0 \leq k < r$ and $i \leq j$ and $B_l = X^k_{i, j+t}$ for some $t \geq 0$ and $1\leq l \leq n$.
\item $A = Y^k_{i,j}$ for some $0 \leq k < r$ and $j \leq i$ and $B_l = Y^k_{i+t, j}$ for some $t \geq 0$ and $1\leq l \leq n$.
\item $A = Z^k_{i,j}$ for some $0 \leq k < r$ and $i, j \in \mathbb{Z}$ and $B_p = Z^k_{i+t,j}$, $B_q = Z^k_{i, j+s}$ for some $t,s \geq 0$ and $1 \leq p,q \leq n$.
\end{enumerate}

We argue, by contradiction, that none of these cases arise. In each of these cases, by inspection of the Hom-hammocks, the set of indecomposable compact objects $C$ for which $F_f(C) \neq 0$ is finite.  By Corollary~\ref{cor:isol}, the open set $(F_f)$ contains only finitely many isolated points of $\Zg(\sK)$.  It follows from Lemma~\ref{lem:KG-defined}, that the CB-rank of $\Zg(\sK)$ is defined and so the isolated points are dense (see, for example, \cite[Lem.~5.3.36]{PreNBK}).  But each isolated point, being of finite endolength, is closed (see \cite[Thm.~5.1.12]{PreNBK}) so there are no other points in $(F_f)$.  Since these are the only points on which $F_f$ is nonzero and since their direct sum is of finite endolength, it follows that $F_f$ is finite length and $q_0(F_f) = 0$ which is a contradiction.
\end{proof}

The completeness statement in Proposition~\ref{prop:CB1pureinj} now follows from the corollary below.

\begin{corollary} 
The simple objects in $\Coh{\sK}_{\cl{X}_0}$ are in natural one-to-one correspondence with the $\sim$-equivalence classes of $1$-simple morphisms. 
\end{corollary}

\subsection{Cantor-Bendixson rank $2$}  

We obtain ${\Coh{\sK}_{\cl{X}_1}}$ by localising $\Coh{\sK}$ at the Serre subcategory consisting of the functors $F$ such that $q_0(F)$ has finite length.  

\begin{remark}\label{rem: finite points}
Since the isolation condition holds and using Corollary~\ref{cor:pure-injective}, we may apply a similar argument to the one contained in the proof of Proposition~\ref{prop: 1-simples} to obtain that an object $q_0(F)$ in $ \Coh{\sK}_{\cl{X}_0}$ is finite length if and only if $(F)$ contains finitely many points of CB-rank $1$.   
\end{remark}

\begin{proposition}
Let $\Lambda$ be a derived-discrete algebra of either finite or infinite global dimension.
An object $q_1(F)$ in $\Coh{\sK}_{\cl{X}_1}$ is simple if and only if $q_1(F) = q_1((Z_j^k,-))$ for some $j,k$.
\end{proposition}

\begin{proof}
We give an argument for the case where $\Lambda$ has infinite global dimension; the case where $\Lambda$ has finite global dimension can be found in \cite{BK}.

Note that, by Remark~\ref{rem: finite points} and the description of Hom-hammocks in Proposition~\ref{prop:inf-hammocks}, the objects ${q_1((X_{i,j}^k, -)) = 0}$ and $q_1((Z_{j}^k, -)) \neq 0$ for all $i, j \in \mathbb{Z}$ and $0 \leq k < r$.  It follows from this and Lemma~\ref{lem:indecdom} that any simple object will be of the form $q_1(F_f)$ where $f \colon Z_j^k \to B$ for some compact object $B$, $j \in \mathbb{Z}$ and $0 \leq k < r$.  It remains to show that $q_1((Z_j^k, -))$ is simple.  Any coherent subobject of $q_1((Z_j^k, -))$ will (by Lemma~\ref{lem:coherent-subfunctors}) be the image under $q_1$ of some $\image(f, -) \subseteq (Z_j^k, - )$ where $f \colon Z_j^k \to \bigoplus_{i=1}^n B_i$ with $B_1,\dots, B_n$ indecomposable.  Note that $\image(f,-)$ is the sum of the $\image(\pi_if,-)$ where $\pi_i$ is the projection to $B_i$.  If $B_i = Z_t^k$ for any $1 \leq i \leq n$ and $t \geq j$, then Remark~\ref{rem: finite points} and Proposition~\ref{prop:inf-hammocks} gives us that $q_1(F_f) = 0$ and so $q_1((Z_j^k, - )) = q_1(\image(f,-))$.  So consider the case where $B_i \in \mcX^{k+1}$ for each $1\leq i\leq n$.  Then there is an epimorphism $(\bigoplus_{i=1}^n B_i, -) \to \image(f,-)$ and so $q_1(\image(f,-)) = 0$.  
\end{proof}

\begin{corollary} \label{cor:CB2pureinj}
Let $\Lambda$ be a derived-discrete algebra of either finite or infinite global dimension.
The set $\mcZ_\infty$ is a complete list of all indecomposable, pure-injective complexes of CB-rank $2$.
\end{corollary}

\begin{proof}
We have already seen that the objects $Z_\infty^k$ are indecomposable and pure-injective and it is clear from the Hom-hammocks that the Hom-functors $(Z_j^k, - )$ isolate these points in $\cl{X}_2$.
\end{proof}

This now completes the proof of Theorem~\ref{thm:KG-dim}.
Although we do not do it here explicitly, using these methods and results, it is possible to obtain a complete description of the topology of the Ziegler spectrum of $\sK$.

\section{Indecomposable complexes with compact support} \label{ind-compact-supp}

Before showing that all indecomposable objects of $\sK$ are pure-injective, we establish some preliminary results which hold for $\Lambda$ of both finite and infinite global dimension. 
For an object $M$ of a compactly generated triangulated category $\T$, its \emph{support} is the Ziegler-closed subset $\supp(M) \coloneqq \Def{M} \cap \Zg(\T)$, where $\Def{M}$ is the definable subcategory generated by $M$; see Definition~\ref{def:definable}.
Recall the definition of localisation at a definable subcategory from Section~\ref{sec:def-localisation}.

\begin{lemma} \label{lem:direct-summand}
Let $M$ be an object in a compactly generated triangulated category $\T$.  If $\supp(M)$ contains a compact object $C$ such that $\{ C \} = (F) \cap \supp(M)$ for some functor $F \in \Coh{\T}$ whose image is simple in $\Coh{\T}_{\supp(M)}$, then $C$ is a direct summand of $M$.
\end{lemma}

\begin{proof}
Let $q \colon \Coh{\T} \to \Coh{\T}_{\supp(M)}$ be the localisation functor, which is left adjoint to the canonical embedding $i \colon \Coh{\T}_{\supp(M)} \into \Coh{\T}$. Recall from Section~\ref{sec:def-localisation}, that there is a parallel and compatible localisation $(\Mod{\Tc})_{\supp(M)}$, whose localisation functor and canonical embedding we denote by $q$ and $i$ again. 

Let $F$ be as in the statement of the lemma.  Since $F$ is nonzero in $\Coh{\T}_{\supp(M)}$, $F(M)\neq 0$ and, since $C\in (F)$, $F(C)\neq 0$. By Lemma~\ref{lem:functors}, $F = G^\vee$ for some functor $G \in \mod{\Tc}$ and $q(G)$ must also be simple in $(\Mod{\Tc})_{\supp(M)}$.  Also by that result, $(G, (-,M)) \neq 0$ and $(G, (-,C)) \neq 0$.

Since $(-,M)$ and $(-,C)$ are torsionfree it follows that $(q(G), q(-,M)) \neq 0$ and $(q(G), (-,C)) \neq 0$, so there are embeddings $k \colon q(G) \into q(-,M)$ and $j \colon q(G) \into q(-,C)$.  Since $C$ is an indecomposable pure-injective, by \cite[Prop~11.1.31]{PreNBK} $q(-,C)$ is an indecomposable injective object of $(\Mod{\Tc})_{\supp(M)}$, thus $q(-,C)$ is the injective hull of $q(G)$ and so $j$ must embed $q(G)$ as the simple socle of $q(-,C)$.  The cokernel of $j$ is finitely presented so, since $q(-,M)$ is fp-injective, there exists some morphism $h \colon q(-,C) \to q(-,M)$ such that $k = hj$.  As $q(-,C)$ has simple essential socle, $h$ must be a monomorphism.

Identifying $h$ with its image under $i$, we can regard $h$ as a monomorphism from $(-,C)$($\cong iq(-,C)$) to $(-,M)$ in $\Mod{\Tc}$.  Since $C$ is compact, Yoneda's lemma says that $h$ must be induced by some $h' \colon C \to M$ in $\T$ and by definition this must be a pure monomorphism.  But $C$ is pure-injective so we must conclude that $h'$ splits and $C$ is a direct summand of $M$.
\end{proof}

\begin{corollary} \label{cor:isolation}
Let $M$ be an indecomposable object of $\sK$ and suppose $C \in \supp(M)$ is a compact object which is isolated in $\supp(M)$. Then $C = M$.
\end{corollary}

\begin{proof} 
By Lemma~\ref{lem:KG-defined},  $\KG(\sK) < \infty$ whence the isolation condition holds for $\Zg(\sK)$ by Lemma~\ref{lem:isolation}.  Therefore there exists a functor $F \in \Coh{\sK}$ such that $\{ C \} = (F) \cap \supp(M)$ and $q(F) \in (\Coh{\sK})_{\supp(M)}$ is simple.  Hence, by Lemma~\ref{lem:direct-summand}, $C$ is a direct summand of $M$ and hence $C = M$. 
\end{proof}

\begin{corollary} \label{cor:support-compact}
Let $M$ be an indecomposable object of $\sK$.  If $C \in \supp(M)$ is compact then $M$ is compact.
\end{corollary}

\begin{proof}
 If $C \in \KbprojL$, then $C$ is isolated in $\Zg(\sK)$ and so is also isolated in $\supp(M)$, so Lemma~\ref{lem:isolation} applies.  Otherwise, without loss of generality, assume that $\supp(M) \subseteq \Zg(\sK)'$, i.e.\  $C$ is a compact but nonperfect object, in particular, $\sK$ must have infinite global dimension.  By Proposition~\ref{prop:CB1pureinj} $C$ has CB-rank 1, hence is isolated in $\Zg(\sK)'$ and hence in $\supp(M)$; now apply Corollary~\ref{cor:isolation}.
 \end{proof}
 
 \section{Indecomposable complexes for infinite global dimension derived-discrete algebras} \label{ind-infinite}
 
Recall notation and the structure of $\add{\strings{\Lambda}}$ from Section~\ref{sec:infinite} in the case $\gldim \Lambda = \infty$. 

\begin{setup} \label{set:infinite}
Let $\Lambda$ be a derived-discrete algebra with $\gldim \Lambda = \infty$. Consider the noncompact, right-infinite string complexes $X^k_{i,\infty}$ forming the lefthand beams of the ladder type AR components $\tilde{\mcZ}^k$. 
For each $0\leq k < r$, there is a sequence of irreducible morphisms
\[ 
\xymatrix{ \cdots \ar[r] & X_{i-1,\infty}^k \ar[r]^-{t_{i-1}^k} & X_{i,\infty}^k \ar[r]^-{t_{i}^k} &X_{i+1,\infty}^k \ar[r] & \cdots
}
\]
We define $X_\infty^k := \bigoplus_{i \in \bZ} X_{i,\infty}^k$ and write $\widetilde{X} \coloneqq \bigoplus_{k = 0}^{r-1} X_\infty^k$.
\end{setup}

In Theorem~\ref{thm:inf-indecomposable} we obtain the pure-injectivity of each indecomposable complex of $\sK$ as a corollary of the fact that $\widetilde{X}$ is $\Sigma$-pure-injective -- a strengthening of the definition of pure-injectivity.

\begin{definition} \label{def:sigma-pure-injective}
An object of $\T$ is \emph{$\Sigma$-pure-injective} if the coproduct $N^{(I)}$ is pure-injective for any (possibly infinite) set $I$. Equivalently, $N$ is $\Sigma$-pure-injective if and only if for each $C \in \Tc$, $\Hom_{\T}(C,N)$ satisfies the descending chain condition on $\End_{\T}(N)$-submodules.
\end{definition}

\begin{remark} \label{sigpirmk}
There are convenient references for the equivalence above and related results in the context of module and functor categories, hence which apply directly to $\Mod{\Tc}$. 
In view of the equivalence $\Pinj{\T} \simeq \Inj{\Mod{\Tc}}$ on page~\pageref{pinj-inj-equivalence} between objects and, Lemma~\ref{lem:functors}, between concepts involving definability, these references thus apply equally to compactly generated triangulated categories.  We will use the facts that any object of finite endolength is $\Sigma$-pure-injective (see \cite[Cor.~4.4.24]{PreNBK}), that a direct sum of finitely many $\Sigma$-pure-injective objects is $\Sigma$-pure-injective (see \cite[Lem.~4.4.26]{PreNBK}), and that if $M$ is $\Sigma$-pure-injective then every object in the definable subcategory generated by $M$ is $\Sigma$-pure-injective (see \cite[Prop.~4.4.27]{PreNBK}).  
\end{remark}

\begin{lemma}\label{lem:dcc-perfect}
Suppose we are in the situation of Setup~\ref{set:infinite}. If $C$ is a perfect string complex, then $\Hom_{\sK}(C, X_\infty^k)$ satisfies the descending chain condition on $\End_{\sK}(X_\infty^k)$-submodules.
\end{lemma}

\begin{proof}
Since $C$ is compact,  $\Hom_{\sK}(C, X^k_\infty) \cong \bigoplus_{i\in \bZ} \Hom_{\sK}(C, X^k_{i,\infty})$. By Proposition~\ref{prop:inf-hammocks}, only finitely many of the $X^k_{i,\infty}$ admit nontrivial morphisms from $C$, making the right-hand side of the isomorphism above a finite direct sum. Moreover, by Proposition~\ref{dimensions}, each  $\Hom_{\sK}(C, X^k_{i,\infty})$ is a finite-dimensional $\kk$-vector space, whence $\Hom_{\sK}(C, X^k_\infty)$ is finite-dimensional.  It follows that $\Hom_{\sK}(C, X^k_\infty)$ satisfies the descending chain condition on $\End_{\sK}(X^k_\infty)$-submodules.
\end{proof}

We now turn to the case that $C$ is a nonperfect compact string complex, i.e. by Remark~\ref{rem:correspondence}, $C$ is a right-infinite string complex.

\begin{lemma}\label{lem:dcc-nonperfect}
Suppose we are in the situation of Setup~\ref{set:infinite}. Suppose $C = Z^\ell_j$ for some $0 \leq \ell < r$ and $j \in \bZ$. Then $\Hom_{\sK}(C, X_\infty^k)$ satisfies the descending chain condition on $\End_{\sK}(X_\infty^k)$-submodules.
\end{lemma}

\begin{proof}
By Proposition~\ref{dimensions}, we have $\dim \Hom_{\sK}(C,X^k_{i,\infty}) \leq 1$. 
By Proposition~\ref{prop:inf-hammocks} there exist $N \in \bZ$ such that $\Hom_{\sK}(C,X^k_{i,\infty}) = 0$ for all $i > N$. Starting with $i = N$, we can use Proposition~\ref{prop:factorisations} to define a family of morphisms $b_i \colon C \to X^k_{i,\infty}$ such that for each $j \geq 1$ we have $b_{i+j} = t^k_{i+j-1} \cdots t^k_i b_i$.
By compactness of $C$ and the at most one-dimensionality of $\Hom_{\sK}(C,X^k_{i,\infty})$, we have $\{b_i \mid b_i \neq 0, \, i \in \bZ\}$ is a basis for the Hom-space $\Hom_{\sK}(C,X^k_\infty)$.

Suppose $M \subseteq \Hom_{\sK}(C,X^k_\infty)$ is an $\End_{\sK}(X^k_\infty)$-submodule. If $b_i \in M$ then $b_{i+1} = t_i b_i \in M$ for each $i \in \bZ$. Therefore, if $M$ is a proper submodule of $\Hom_{\sK}(C,X^k_\infty)$ the set $\{i \mid b_i \in M\}$ has a minimal element, so the dimension of $M$ is finite (namely $N-i+1$) and the result follows.
\end{proof}

\begin{proposition} \label{prop:sigma-pure-injective}
Suppose we are in the situation of Setup~\ref{set:infinite}. Then the direct sum of all left-infinite string complexes,
$\widetilde{X} := \bigoplus_{k=0}^{r-1} X^k_\infty$,
is $\Sigma$-pure-injective.
\end{proposition}

\begin{proof}
By Lemmas~\ref{lem:dcc-perfect} and \ref{lem:dcc-nonperfect}, $\Hom_{\sK}(C, X_\infty^k)$ satisfies the descending chain condition for each $0 \leq k < r$ and each indecomposable compact object $C$ of $\sK$. Since the functor $\Hom_{\sK}(-,X^k_\infty)$ commutes with finite direct sums, it follows that $\Hom_{\sK}(C,X_\infty^k)$ satisfies the descending chain condition for each compact object $C$ of $\sK$. By the equivalent formulation of $\Sigma$-pure-injectivity in Definition~\ref{def:sigma-pure-injective}, it follows that $X^k_\infty$ is $\Sigma$-pure-injective. By Remark~\ref{sigpirmk}, $\widetilde{X}$ is $\Sigma$-pure-injective.
\end{proof}

\begin{corollary} \label{cor:noncompact}
Let $\Lambda$ be a derived-discrete algebra with $\gldim \Lambda = \infty$.  If $M$ is an indecomposable object of $\sK$ such that $\supp(M)$ contains only noncompact objects, then $M$ is $\Sigma$-pure-injective.
\end{corollary}

\begin{proof}
Let $Z \coloneqq \widetilde{X} \oplus \big( \bigoplus_{k = 0}^{r-1} Z_\infty^k \big)$ be the direct sum of all noncompact objects in the Ziegler spectrum $\Zg(\sK)$, see Remark~\ref{rem:correspondence}. Then $Z$ is $\Sigma$-pure-injective since it is a finite direct sum of $\Sigma$-pure-injective objects (by Proposition~\ref{prop:sigma-pure-injective} and \ref{cor:pure-injective}).  Then the definable subcategory $\Def{Z}$ contains only $\Sigma$-pure-injective objects.    But by assumption $\Def{M} \subseteq \Def{Z}$ and therefore $M$ is $\Sigma$-pure-injective. 
\end{proof}

Putting this together gives Theorem~\ref{intro:indecomposable} in the infinite global dimension case.

\begin{theorem} \label{thm:inf-indecomposable}
Suppose $\Lambda$ is a derived-discrete algebra with $\gldim \Lambda = \infty$. The objects of $\strings{\Lambda}$ form a complete list of indecomposable objects of $\sK$. In particular, each indecomposable object of $\sK$ is pure-injective.
\end{theorem}

\begin{proof}
Suppose $M$ is an indecomposable object of $\sK$. If $\supp(M)$ contains a compact pure-injective then $M$ itself is compact by Corollary~\ref{cor:support-compact} and therefore pure-injective; see Proposition~\ref{compact} and Remark~\ref{rem:correspondence}. So we may assume that $\supp(M)$ contains no compact pure-injective objects, whence Corollary~\ref{cor:noncompact} tells us that $M$ is in fact $\Sigma$-pure-injective. Now Theorem~\ref{thm:KG-dim} tells us that $\strings{\Lambda}$ is a complete list of indecomposable, pure-injective complexes in $\sK$.
\end{proof}

 \section{Indecomposable complexes for finite global dimension derived-discrete algebras} \label{ind-finite}

In this section, we complete the proof that all indecomposable complexes in $\sK$ for a derived-discrete algebra $\Lambda$ are pure-injective by treating the finite global dimension case. We refer to Section~\ref{sec:finite} for the structure of the AR quiver whose vertices are the indecomposable pure-injective complexes. We start by showing that any indecomposable complex whose support contains only pure-injective complexes of CB rank $2$ is already on our list.

\begin{lemma} \label{lem:CB-rank2}
Let $M$ be an indecomposable object in $\sK$.  If every indecomposable pure-injective in $\supp(M)$ has Cantor-Bendixson rank $2$, then $M$ is $\Sigma$-pure-injective.
\end{lemma}

\begin{proof}
By Corollary~\ref{cor:CB2pureinj}, there are only finitely many pure-injective objects with Cantor-Bendixson rank $2$. By Corollary~\ref{cor:pure-injective} each has finite endolength, hence is $\Sigma$-pure-injective.  Thus, $\supp(M)$ is a finite set in which each object is $\Sigma$-pure-injective. 

Let $N$ be the direct sum of one copy of each object of $\supp(M)$. Then $N$ is $\Sigma$-pure-injective, hence every object in the definable subcategory that $N$ generates, in particular $M$, is $\Sigma$-pure-injective by Remark~\ref{sigpirmk}.
\end{proof}

We shall invoke the following setup.

\begin{setup} \label{set:finite}
Let $\cl{X}_0 \subset \Zg(\sK)$ be the closed set of non-isolated points and consider the localisation adjoint pair from Section~\ref{sec:def-localisation}, 
\[
\xymatrix{
 \Mod{\Kc} \ar@<1ex>[rr]^-{q} & & (\Mod{\Kc})_{\cl{X}_0} \ar@{_{(}->}@<1ex>[ll]^-{i},
}
\]
where the localisation functor $q$ is left adjoint to the canonical inclusion $i$.

Let $M$ be an indecomposable complex in $\sK$. By Corollary~\ref{cor:support-compact} and Lemma~\ref{lem:CB-rank2}, we may assume that $\supp(M)$ contains no compact objects and at least one object of CB rank $1$. Let $N \in \supp(M)$ be such a complex of CB rank $1$. By Section~\ref{sec:finite} and Proposition~\ref{prop:CB1pureinj}, $N$ lies in one of the $A_\infty^\infty$ components of $\add{\strings{\Lambda}}$. If $N \in \mcX_\infty \cup \mcY_\infty$ there is a sequence of indecomposable complexes and irreducible morphisms
\begin{equation} \label{eq:chain}
N = N_1 \rightlabel{\alpha_1} N_2 \rightlabel{\alpha_2} N_3 \rightlabel{\alpha_3} N_4 \rightlabel{\alpha_4} \cdots.
\end{equation}
We shall write $\beta_n = \alpha_n \cdots \alpha_2 \alpha_1$ for the $n$-fold successive composition of the $\alpha_i$.
If $N \in \mcX_{-\infty} \cup \mcY_{-\infty}$ there is a sequences of indecomposable complexes and morphisms
\begin{equation} \label{eq:L-chain}
N \rightlabel{\gamma_0} L_1 \rightlabel{\alpha_1} L_2 \rightlabel{\alpha_2} L_3 \rightlabel{\alpha_3} L_4 \rightlabel{\alpha_4} \cdots.
\end{equation}
in which the $L_i$ lie the corresponding $\mcY_\infty$ or $\mcX_\infty$ component admitting nontrivial maps from $N$, and in which the $\alpha_i$ are irreducible morphisms. We shall write $\gamma_n = \alpha_n \cdots \alpha_1 \gamma_0$.

The sequence \eqref{eq:chain} induces the following sequence in $\Mod{\sK^c}$,
\begin{equation*}
(-,N) = (-,N_1) \rightlabel{(-,\alpha_1)} (-,N_2) \rightlabel{(-,\alpha_2)} (-,N_3) \rightlabel{(-,\alpha_3)} (-,N_4) \rightlabel{(-,\alpha_4)} \cdots.
\end{equation*}
Likewise, the sequence \eqref{eq:L-chain} induces the following sequence in $\Mod{\sK^c}$.
\begin{equation*}
(-,N) \rightlabel{(-,\gamma_0)} (-,L_1) \rightlabel{(-,\alpha_1)} (-,L_2) \rightlabel{(-,\alpha_2)} (-,L_3) \rightlabel{(-,\alpha_3)} \cdots.
\end{equation*}
\end{setup}

\begin{remark}\label{lem:M-is-torsionfree}
If $M \in \ind{\sK}$ is such that $\supp(M)$ does not contain any isolated points, then $(-,M) \cong i \circ q(-,M)$.  This follows from \cite[Cor.~12.3.3]{PreNBK}, though here the statement is simplified since, as described in Section \ref{sec:def-localisation}, we can work in $\Mod(\Tc)$ rather than in an associated functor category.
\end{remark}

\begin{lemma} \label{lem:subfunctors}
In the situation of Setup~\ref{set:finite}, we have the following sequence of subfunctors of $(-,N)$,
\[
\ker(-,\beta_1) \into \ker(-,\beta_2) \into  
\cdots
\quad \text{or} \quad
\ker(-,\gamma_0) \into \ker(-,\gamma_1) \into  
\cdots
\]
such that $\bigcup_{n\geq 1} \ker(-,\beta_n) = (-,N)$ or $\bigcup_{n\geq 0} \ker(-,\gamma_n) = (-,N)$.
\end{lemma}

\begin{proof}
It is enough to check the statement on compact objects lying in $\infhomto{N}$ (see Section~\ref{hammocks-notation}). There are essentially two cases: $N$ lies in an $\mcX_\infty$ or $\mcY_\infty$ component, or else $N$ lies in an $\mcX_{-\infty}$ or $\mcY_{-\infty}$ component. One can then check the statement by comparing the relevant Hom-hammocks, see Figure~\ref{kernel-seq} for the relevant diagrams, where the top diagram refers to the case that $N$ lies in an $\mcX_\infty$ or $\mcY_\infty$ component and the bottom diagram to the case $N$ lies in an $\mcX_{-\infty}$ or $\mcY_{-\infty}$ component.
\end{proof}

\begin{figure}
\begin{center}
\begin{tikzpicture}[scale=0.2,thick]

\newcommand{\sst}[1]{\scalebox{0.75}{$#1$}}

  \draw[fill,blue!10] (33,-4) -- (26,-11) -- (24,-9) -- (31,-2) -- cycle;
  \draw[fill,blue!10] (22,-7) -- (13, 2) -- (20, 9) -- (29,0) -- cycle;
  \draw[fill,blue!30] (31,-6) -- (26,-11) -- (24,-9) -- (29, -4) -- cycle;
  \draw[fill,blue!30] (22,-7) -- (13,2) -- (18,7) -- (27,-2) -- cycle;
  \draw[fill,blue!50] (29,-8) -- (26, -11) -- (24,-9) -- (27,-6) -- cycle;
  \draw[fill,blue!50] (22,-7) -- (13, 2) -- (16, 5) -- (25, -4) -- cycle;
  
\draw[blue,dashed] (36,-7) -- (32,-11) -- (27,-6);   
\draw[blue,dashed] (25,-4) -- (16,5);

\draw[blue, dotted] (38,-9) -- (36,-11) -- (29,-4);
\draw[blue,dotted] (27,-2) -- (18,7);

\draw[dashed] (-2,9) -- (7,0) -- (-2,-9);

\draw[gray!80] (8,1) -- (-1,10);
\draw[gray!80] (8,-1) -- (-1,-10); 

  \draw[dashed] (0, 11) -- (9,  2) -- (18, 11);
  \draw         (0, 11) -- (18, 11);

\draw[gray!80] (10, 1) -- (19,10);
\fill[gray!80] (9,0) circle (0.2);
\draw[gray!80] (10,-1) -- (19,-10);

  \draw[dashed] (0,-11) -- (9, -2) -- (18,-11);
  \draw         (0,-11) -- (18,-11);

  \draw[dashed] (11, 0) -- (20, 9) -- (29, 0) -- (20,-9) -- cycle;

  \draw[gray!80] (21, 10) -- (30,1);
  \draw[gray!80] (21,-10) -- (30,-1);
  
 \draw[dashed] (22, 11) -- (31,  2) -- (40, 11);
 \draw         (22, 11) -- (40, 11);

  \draw[dashed] (22,-11) -- (31, -2) -- (40,-11);
  \draw         (22,-11) -- (40,-11);

  \draw[gray!80] (32, 1) -- (41,10);
  \fill[gray!80] (31,0) circle (0.2);
  \draw[gray!80] (32,-1) -- (41,-10);
  \fill (34, -3) circle (0.2);
  \fill (37, -6) circle (0.2);
  \fill (39, -8) circle (0.2);
  
  \node at (36, -1) {\sst{N}};
  \node at (39, -4) {\sst{N_t}};
  \node at (41, -6) {\sst{N_s}};
   
  \draw[dashed] (42,-9) -- (33, 0) -- (42, 9);
  
\end{tikzpicture}
\end{center}
\vspace{5mm}
\begin{center}
\begin{tikzpicture}[scale=0.2,thick]

\newcommand{\sst}[1]{\scalebox{0.75}{$#1$}}

\fill[red!10] (26,3) -- (20,9) -- (11,0) -- (17,-6) -- cycle;
\fill[red!50] (15,-8) -- (12,-11) -- (6, -5) -- (9,-2) -- cycle;

\fill[red!50] (19,-4) -- (13,2) -- (11,0) -- (17,-6) -- cycle;
\fill[red!30] (23,0) -- (17,6) -- (13,2) -- (19,-4) -- cycle;

\draw[red,dashed] (33,-4) -- (26,-11) -- (24,-9);
\draw[red,dashed] (22,-7) -- (13,2);

\draw[red,dotted] (37,-8) -- (34,-11) -- (28,-5);
\draw[red,dotted] (26,-3) -- (17,6);

\draw[dashed] (-2,9) -- (7,0) -- (-2,-9);

\draw[gray!80] (8,1) -- (-1,10);
\draw[gray!80] (8,-1) -- (-1,-10); 

 \draw[dashed] (0, 11) -- (9,  2) -- (18, 11);
 \draw         (0, 11) -- (18, 11);

 \draw[gray!80] (10, 1) -- (19,10);
 \fill[gray!80] (9,0) circle (0.2);
 \draw[gray!80] (10,-1) -- (19,-10);

  \draw[dashed] (0,-11) -- (9, -2) -- (18,-11);
  \draw         (0,-11) -- (18,-11);

  \draw[dashed] (11, 0) -- (20, 9) -- (29, 0) -- (20,-9) -- cycle;

  \draw[gray!80] (21, 10) -- (30,1);
  \fill (27,4) circle (0.2);
  \node at (29,6) {\sst{N}};
  \draw[gray!80] (21,-10) -- (30,-1);
  
\draw[dashed] (22, 11) -- (31,  2) -- (40, 11);
\draw         (22, 11) -- (40, 11);

  \draw[dashed] (22,-11) -- (31, -2) -- (40,-11);
  \draw         (22,-11) -- (40,-11);

  \draw[gray!80] (32, 1) -- (41,10);
  \fill[gray!80] (31,0) circle (0.2);
  \draw[gray!80] (32,-1) -- (41,-10);
  \fill (34, -3) circle (0.2);
  \fill (38, -7) circle (0.2);
  
  \node at (36, -1) {\sst{L_1}};
  \node at (40, -5) {\sst{L_t}};

\draw[dashed] (42,-9) -- (33, 0) -- (42, 9);    
\end{tikzpicture}
\end{center}
\caption{\textit{Above:} the situation in \eqref{eq:chain}. The darkest shaded region shows the objects $C$ such that $(C,N) \neq 0$ but where there are nontrivial maps $f\colon C \to N$ such that $f\in \ker(C,\beta_{t+1})$. The intermediate shaded region together with the darkest shaded region shows those objects such that $(C,N) \neq 0$, and where there are nontrivial maps $f\colon C \to N$ such that $f \in \ker(C,\beta_{s+1})$. The three shaded regions show objects $C$ such that $(C,N) \neq 0$ The dashed and dotted regions show objects $C$ such that $(C,N_t) \neq 0$ and $(C,N_s) \neq 0$, respectively. 
\textit{Below:} the situation in \eqref{eq:L-chain}. The regions are shaded as in the sketch above with $\gamma$s playing the role of $\beta$s.} \label{kernel-seq}
\end{figure}
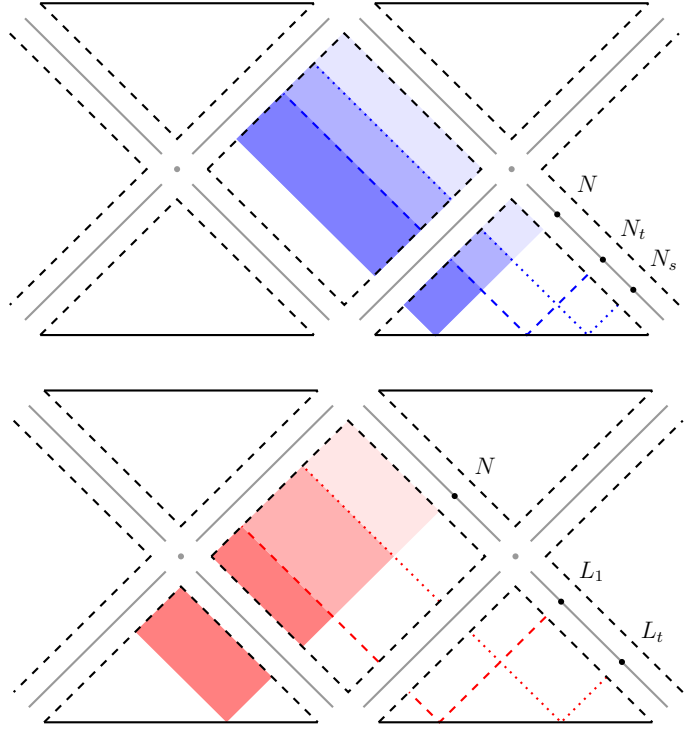

\begin{lemma} \label{lem:socle}
In the situation of Setup~\ref{set:finite}, the functor $q(-,N_i)$ is injective with simple socle for each $i \geq 1$. 
\end{lemma}

\begin{proof}
By Proposition~\ref{prop:CB1pureinj}, each $N_i$ has CB-rank $1$. Thus, by Proposition~\ref{prop:KG-CB}, there is a functor $S_i' \in \Coh{\sK}$ whose image is simple in $\Coh{\sK}/\Coh{\sK}_0 = \Coh{\sK}_{\cl{X}_0}$ and such that $N_i$ is the unique point of CB-rank 1 in the open set $(S_i')$.  Then we continue as in the proof of Lemma~\ref{lem:direct-summand}.
\end{proof}

\begin{lemma} \label{lem:kernel-quotients}
In the situation of Setup~\ref{set:finite}, for $n \geq 1$ we have
\[
\dfrac{q \ker(-,\beta_n)}{q \ker(-,\beta_{n-1})} \cong q \ker (-,\alpha_n)
\quad \text{or} \quad
\dfrac{q \ker(-,\gamma_n)}{q \ker(-,\gamma_{n-1})} \cong q \ker (-,\alpha_n)
\]
and $q \ker (-,\alpha_n)$ is a simple functor in $(\Mod{\Kc})_{\cl{X}_0}$.
\end{lemma}

\begin{proof}
For the first claim, observe that there are maps
\[
\dfrac{\ker(-,\beta_n)}{\ker(-,\beta_{n-1})} \rightlabel{\phi} \ker(-,\alpha_n) 
\quad \text{or} \quad 
\dfrac{\ker(-,\gamma_n)}{\ker(-,\gamma_{n-1})} \rightlabel{\phi} \ker(-,\alpha_n).
\]
In each case, consider the inclusion $\image \phi \into \ker(-,\alpha_n)$.
Inspecting the Hom-hammocks, we see that $\ker(-,\alpha_n)/\image \phi$ is nonzero on only $n-1$ compact objects,
namely in the first case the points lying on the blue-dashed coray from the shaded region to the boundary of the AR component in the top diagram of Figure~\ref{kernel-seq}, and in the second case the points lying on a coray essentially similar to that of the red-dotted coray from its intersection with the red-dashed ray to the boundary of the AR component in the bottom diagram of Figure~\ref{kernel-seq}. 
It therefore follows that since $\ker(-,\alpha_n)/\image \phi$ is nonzero on only finitely many points that it is a finite length functor, cf. Remark~\ref{rem: finite points}. Hence $q \ker(-,\alpha_n)/\image \phi = 0$, giving the desired isomorphism.

For the second claim, note that in light of Lemma~\ref{lem:functors} and Corollary~\ref{cor:functors-explicit}, we need to exhibit a family $\mcH$ of equivalent $1$-simple morphisms (see Definition~\ref{def:1-simple}) such that for each compact object $C$ there exists $h \in \mcH$ and an embedding $\ker(C,\alpha_n) \rightarrow \ker(C,h)$.  Since $q \ker(-, \alpha_n) \neq 0$, it will follow that $q \ker(-, \alpha_n) = q \ker(-, h)$ and hence that $q \ker(-, \alpha_n)$ is simple.

There are four cases for $\mcH$ depending on which type of component $N$ sits in, where the reader should recall Notation~\ref{notation}:
\begin{itemize}
\item if $N = X^{k}_{a,\infty}$ take $\mcH = \{Z^{k+1}_{a,j} \to Z^{k+1}_{a+1,j} \oplus X^{k+1}_{t,a'-1} \mid j \in \bZ \text{ and } t \leq a'-1 \}$; 
\item if $N = Y^{k}_{\infty,b}$ take $\mcH = \{Z^{k+1}_{i,b} \to Z^{k+1}_{i,b+1} \oplus Y^{k+1}_{b''-1,t} \mid i \in \bZ \text{ and } t \geq b''-1 \}$; 
\item if $N = X^{k}_{-\infty,b}$  and $L_1 = Y^k_{\infty,d}$ take $\mcH = \{Z^{k}_{i,d} \to Z^{k}_{i,d+1} \oplus Y^{k+1}_{d''-1,t} \mid i \in \bZ \text{ and } t \geq d''-1 \}$;
\item if $N = Y^{k}_{a,-\infty}$  and $L_1 = X^k_{c,\infty}$ take $\mcH = \{Z^{k}_{c,j} \to Z^{k}_{c+1,j} \oplus X^{k+1}_{t,c'-1} \mid i \in \bZ \text{ and } t \leq c'-1 \}$. \qedhere
\end{itemize}
\end{proof}

We are now ready to prove the main theorem of this section, which is Theorem~\ref{intro:indecomposable} in the finite global dimension case.

\begin{theorem} \label{thm:finite-indecomposable}
Suppose $\Lambda$ is a derived-discrete algebra such that $\gldim \Lambda < \infty$. The objects of $\strings{\Lambda}$ form a complete list of indecomposable objects of $\sK$. In particular, each indecomposable object of $\sK$ is pure-injective.
\end{theorem}

\begin{proof}
Theorem~\ref{thm:KG-dim} tells us that $\strings{\Lambda}$ is a complete list of indecomposable, pure-injective objects of $\sK$. We need to check that there are no further indecomposable objects. 

Suppose, for a contradiction, $M \in \ind{\sK}$ is not a string complex. By Corollary~\ref{cor:support-compact} and Lemma~\ref{lem:CB-rank2}, we may assume that $M$ satisfies the hypotheses of Setup~\ref{set:finite}. By Lemma~\ref{lem:socle}, there is an embedding $q S \into q (-,N)$ where $qS$ is simple and, arguing as in the proof of Lemma~\ref{lem:direct-summand}, there is also an embedding $q S \into q(-,M)$. Indeed, we have $q S = q \ker(-,\alpha_1) = q \ker(-,\beta_1)$. Lemma~\ref{lem:kernel-quotients} now allows us to repeatedly use the fp injectivity of $q(-,M)$ to get that the map $q S \into q(-,M)$ lifts along the chain of subfunctors of Lemma~\ref{lem:subfunctors}:
\[
\xymatrix{
qS = q \ker(-,\beta_1) \ar@{^{(}->}[r] \ar@{^{(}->}[d] & q \ker(-,\beta_2) \ar@{^{(}->}[r] \ar@{_{(}-->}[dl]_{\exists} & q \ker(-,\beta_3) \ar@{^{(}->}[r] \ar@{_{(}-->}[dll]_{\exists} & \cdots ,\\
q(-,M)                                                                        &                                                                                           &                                                                                       &               }
\]  
giving rise to a map $q(-,N) \to q(-,M)$, since by Lemma~\ref{lem:subfunctors}, $\colim q \ker(-,\beta_i) = q(-,N)$. Now using Lemma~\ref{lem:M-is-torsionfree}, we can deduce that there is a nontrivial map $(-,N) \into (-,M)$. Pure-injectivity of $N$ now means that $(-,N) \into (-,M)$ is a split monomorphism.
Thus, we have a pure epimorphism  $f \colon M\rightarrow N$ such that the induced $(-,f) \colon (-,M)\rightarrow (-,N)$ is split by a morphism $\tau \colon (-,N) \rightarrow (-,M)$, that is, $(-,f)\tau =1_{(-,N)}$.  We want to find some $h:N\rightarrow M$ with $fh=1_N$.

The next part of the argument requires some homotopy string combinatorics; see Section~\ref{structure}.
Since $N$ lies in an $\mcX^k_{\pm \infty}$ or $\mcY^k_{\pm \infty}$ component, $N$ is a one-sided string complex (Definition~\ref{string-complexes}).
 For convenience of exposition, we assume that $N = P_w$ is a right-infinite string complex, where $w$ is a right-infinite subword of $w_\infty$. Moreover, we may assume without loss of generality that the projective $\Lambda$-module sitting in cohomological degree $n$, $N^n$, is zero for all $n <0$ and $N^0 \neq 0$. An analogous argument holds in the case that $N$ is a left-infinite string complex.

There is a sequence of finite subwords of $w$ corresponding to a sequence of perfect complexes $C_i = P_{v_i}$ in which $C_i^n = 0$ for each $n < 0$ and $C_i^0 \neq 0$, 
\begin{equation} \label{eq:subword-seq}
v_0 \subset v_1 \subset v_2 \subset \cdots \subset w 
\quad \longleftrightarrow \quad
C_0 \rightlabel{\alpha_0} C_1 \rightlabel{\alpha_1} C_2 \rightlabel{\alpha_2} \cdots \to N,
\end{equation}
together with a graph map $\iota_i \colon C_i \to N$, such that each $\alpha_i$ is a graph map and each morphism $\iota_i \colon C_i \to N$ factors as
\[
\xymatrix{
C_i \ar[r]^{\alpha_i} \ar[dr]_{\iota_i} & C_{i+1}. \ar[d]^{\iota_{i+1}} \\
                                     & N }
\] 
We remark that the sequence~\eqref{eq:subword-seq} is a subsequence of indecomposable complexes lying on an appriopriately chosen ray or coray in an $\mcX^k$, $\mcY^k$ or $\mcZ^k$ component.

Define $\length \iota_i = \max \{ n \mid (\iota_i)^n \neq 0 \}$. In particular, this means that $(\iota_i)^n = \id_{C^n} = \id_{N^n}$ for all $n \leq \length \iota_i - 1$ by \cite[\S 3]{ALP}.

Applying $(-,f)\colon (-,M) \to (-,N)$ to each $\alpha_i \colon C_i \to C_{i+1}$ gives a commutative tower in which $\tau_i \colon (C_i,N) \to (C_i,M)$ is the splitting map: 
\[
\xymatrix{
(C_0,M) \ar@/_/[r]_{(C_0,f)}                     & (C_0,N) \ar@/_/[l]_{\tau_0}                     & \backepsilon \iota_0        \\
(C_1,M) \ar@/_/[r]_{(C_1,f)} \ar[u]^{(\alpha_0,M)} & (C_1,N) \ar@/_/[l]_{\tau_1} \ar[u]_{(\alpha_0,N)} & \backepsilon \iota_1 \ar[u] \\
(C_2,M) \ar@/_/[r]_{(C_2,f)} \ar[u]^{(\alpha_1,M)} & (C_2,N) \ar@/_/[l]_{\tau_2} \ar[u]_{(\alpha_1,N)} & \backepsilon \iota_2 \ar[u] \\
\ar@{.}[u]                                     & \ar@{.}[u]                                       &
}
\]
From the commutativity of the squares in the tower above, we get
\[
\xymatrix{
C_i \ar[r]^{\alpha_i} \ar[dr]_{\tau_i(\iota_i)} & C_{i+1} \ar[d]^{\tau_i(\iota_{i+1})} \\ 
                                            & M 
}
\]

We define a map $h \colon N \to M$ componentwise $h^n \colon N^n \to M^n$ by
\[
h^n \coloneqq h_i^n \coloneqq (\tau_i(\iota_i))^n \circ ((\iota_i)^n)^{-1} \text{ for all } n < \length \iota_i ,
\] 
noting that since $\iota_i$ is a graph map that $(\iota_i)^n$ is an isomorphism for each $n < \length \iota_i$. We first need to check that this is well defined, i.e. if $j \geq i$ then $h_i^n = h_j^n$ for $n < \length \iota_i$. Writing $\alpha_{ij} \coloneqq \alpha_{j-1} \cdots \alpha_i$, observe that $\length \alpha_{ij} = \length \alpha_i = \length \iota_i$ and that for $n < \length \alpha_i$, the $n^{\mathrm{th}}$-component $(\alpha_{ij})^n$ is an isomorphism. Hence,
\begin{eqnarray*}
h_j^n & = &(\tau_j(\iota_j))^n \circ ((\iota_j)^n)^{-1} \\
      & = & \big( (\tau_j(\iota_j))^n \circ (\alpha_{ij})^n\big) \circ \big( ((\alpha_{ij})^n)^{-1} \circ ((\iota_j)^n)^{-1} \big) \\
      & = &(\tau_i(\iota_i))^n \circ ((\iota_i)^n)^{-1} = h_i^n.
\end{eqnarray*}

Next, we have to show that $h$ is a cochain map. By choosing $i$ sufficiently large, it is clear that $h^{n+1} d^n_N = d^n_M h^n$ for all $n < \length \iota_i - 1$, in particular, $h$ is a cochain map. 
Note that, in light of what follows, $h$ cannot be null-homotopic, for if it were, $N$ itself would be a null-homotopic giving a contradiction.

Finally, we show that $fh=\id_N$, for which it is enough to show that, for each $n$, we have $f^n \circ h^n= \id_{N^n}$.  By the splitting we have $(C_i,f)\big(\tau_i(\iota_i) \big) = \iota_i$, that is, $f \circ \tau_i(\iota_i) = \iota_i$ which, in the $n^{\mathrm{th}}$ component, becomes $f^n \circ (\tau_i(\iota_i))^n = (\iota_i)^n$, whence $f^n \circ (\tau_i(\iota_i))^n \circ ((\iota_i)^n)^{-1} = \id_{N^n}$, that is $f^n \circ h^n= \id_{N^n}$, as required. We have therefore found a split monomorphism $h \colon N \into M$, hence $M = N$.
\end{proof}

\begin{remark}
We finish by remarking that, although every indecomposable complex over a derived-discrete algebra is pure-injective, there are complexes which are neither pure-injective nor a direct sum of indecomposable objects.  One may see this directly or use \cite[Thm.~9.3]{Belig} or \cite[Thm.~2.10]{KraTel}.
\end{remark}

\bibliography{bibliography}{}

\begin{thebibliography}{10}

\bibitem{ALP}
K.~K. Arnesen, R.~Laking, and D.~Pauksztello.
\newblock Morphisms between indecomposable complexes in the bounded derived
  category of a gentle algebra.
\newblock {\em \harxiv{1411.7644}}, 2014.

\bibitem{Auslander}
M.~Auslander.
\newblock Representation theory of {A}rtin algebras. {II}.
\newblock {\em Comm. Algebra}, 1:269--310, 1974.

\bibitem{BM}
V.~Bekkert and H.~A. Merklen.
\newblock Indecomposables in derived categories of gentle algebras.
\newblock {\em Algebr. Represent. Theory}, 6(3):285--302, 2003.

\bibitem{Belig}
A.~Beligiannis.
\newblock Relative homological algebra and purity in triangulated categories.
\newblock {\em J. Algebra}, 227(1):268--361, 2000.

\bibitem{BenGnac}
D.~J. Benson and G.~Ph. Gnacadja.
\newblock Phantom maps and purity in modular representation theory. {I}.
\newblock {\em Fund. Math.}, 161(1-2):37--91, 1999.
\newblock Algebraic topology (Kazimierz Dolny, 1997).

\bibitem{Bo}
G.~Bobi{\'n}ski.
\newblock The almost split triangles for perfect complexes over gentle
  algebras.
\newblock {\em J. Pure Appl. Algebra}, 215(4):642--654, 2011.

\bibitem{Bo2}
G.~Bobi{\'n}ski.
\newblock The graded centers of derived discrete algebras.
\newblock {\em J. Algebra}, 333:55--66, 2011.

\bibitem{BGS}
G.~Bobi{\'n}ski, C.~Gei{\ss}, and A.~Skowro{\'n}ski.
\newblock Classification of discrete derived categories.
\newblock {\em Cent. Eur. J. Math.}, 2(1):19--49 (electronic), 2004.

\bibitem{BK}
G.~Bobi{\'n}ski and H.~Krause.
\newblock The {K}rull-{G}abriel dimension of discrete derived categories.
\newblock {\em Bull. Sci. Math.}, 139(3):269--282, 2015.

\bibitem{BPP1}
N.~Broomhead, D.~Pauksztello, and D.~Ploog.
\newblock Discrete derived categories {I}: {H}omomorphisms, autoequivalences
  and t-structures.
\newblock {\em \harxiv{1312.5203}}, 2013.

\bibitem{BPP2}
N.~Broomhead, D.~Pauksztello, and D.~Ploog.
\newblock Discrete derived categories {II}: {T}he silting pairs {CW} complex
  and the stability manifold.
\newblock {\em J. London Math. Soc.}, in press, 2016.

\bibitem{BD}
I.~Burban and Y.~Drozd.
\newblock Derived categories of nodal algebras.
\newblock {\em J. Algebra}, 272(1):46--94, 2004.

\bibitem{CBLoc}
W.~W. Crawley-Boevey.
\newblock Locally finitely presented additive categories.
\newblock {\em Comm. Algebra}, 22(5):1641--1674, 1994.

\bibitem{Geigle}
W.~Geigle.
\newblock The {K}rull-{G}abriel dimension of the representation theory of a
  tame hereditary {A}rtin algebra and applications to the structure of exact
  sequences.
\newblock {\em Manuscripta Math.}, 54(1-2):83--106, 1985.

\bibitem{HanZhang}
Y.~Han and C.~Zhang.
\newblock Brauer-{T}hrall type theorems for derived category.
\newblock {\em \harxiv{1310.2777}}, 2013.

\bibitem{HanThes}
Z.~Han.
\newblock {\em The homotopy categories of injective modules of derived discrete
  algebras}.
\newblock PhD thesis, Bielefeld University, 2013.

\bibitem{Jorgensen}
P.~J{\o}rgensen.
\newblock The homotopy category of complexes of projective modules.
\newblock {\em Adv. Math.}, 193(1):223--232, 2005.

\bibitem{KraGen}
H.~Krause.
\newblock Generic modules over {A}rtin algebras.
\newblock {\em Proc. London Math. Soc. (3)}, 76(2):276--306, 1998.

\bibitem{KraDec}
H.~Krause.
\newblock Decomposing thick subcategories of the stable module category.
\newblock {\em Math. Ann.}, 313(1):95--108, 1999.

\bibitem{KraTel}
H.~Krause.
\newblock Smashing subcategories and the telescope conjecture---an algebraic
  approach.
\newblock {\em Invent. Math.}, 139(1):99--133, 2000.

\bibitem{KraStab}
H.~Krause.
\newblock Coherent functors in stable homotopy theory.
\newblock {\em Fund. Math.}, 173(1):33--56, 2002.

\bibitem{Neeman-Groth}
A.~Neeman.
\newblock The {G}rothendieck duality theorem via {B}ousfield's techniques and
  {B}rown representability.
\newblock {\em J. Amer. Math. Soc.}, 9(1):205--236, 1996.

\bibitem{NeeReview}
A.~Neeman.
\newblock Review of {D}. {B}ensen, {G}. {P}h. {G}nacadja: `{P}hantom maps and
  purity in modular representation theory. {I}'.
\newblock {\em Mathematical Reviews MR1713200}, 1999.

\bibitem{Neeman-book}
A.~Neeman.
\newblock {\em Triangulated categories}, volume 148 of {\em Annals of
  Mathematics Studies}.
\newblock Princeton University Press, Princeton, NJ, 2001.

\bibitem{Neeman}
A.~Neeman.
\newblock The homotopy category of flat modules, and {G}rothendieck duality.
\newblock {\em Invent. Math.}, 174(2):255--308, 2008.

\bibitem{PrestMTM}
M.~Prest.
\newblock {\em Model theory and modules}, volume 130 of {\em London
  Mathematical Society Lecture Note Series}.
\newblock Cambridge University Press, Cambridge, 1988.

\bibitem{PreNBK}
M.~Prest.
\newblock {\em Purity, spectra and localisation}, volume 121 of {\em
  Encyclopedia of Mathematics and its Applications}.
\newblock Cambridge University Press, Cambridge, 2009.

\bibitem{PreCatImag}
M.~Prest.
\newblock Categories of imaginaries for definable additive categories.
\newblock {\em \harxiv{1202.0427}}, 2012.

\bibitem{Qin}
Y.~Qin.
\newblock Jordan-{H}\"older theorems for derived categories of derived discrete
  algebras.
\newblock {\em \harxiv{1506.08266}}, 2015.

\bibitem{Stenstrom}
B.~Stenstr{\"o}m.
\newblock {\em Rings of quotients}.
\newblock Springer-Verlag, New York-Heidelberg, 1975.
\newblock Die Grundlehren der Mathematischen Wissenschaften, Band 217, An
  introduction to methods of ring theory.

\bibitem{Vossieck}
D.~Vossieck.
\newblock The algebras with discrete derived category.
\newblock {\em J. Algebra}, 243(1):168--176, 2001.

\bibitem{Ziegler}
M.~Ziegler.
\newblock Model theory of modules.
\newblock {\em Ann. Pure Appl. Logic}, 26(2):149--213, 1984.

\end{thebibliography}
\bibliographystyle{plain}

\end{document}